\documentclass[12pt]{amsart} 

\usepackage{amsthm} 
\usepackage{amsmath} 
\usepackage{amssymb} 

\usepackage{microtype} 
\usepackage{pinlabel} 

\usepackage[
pdfborderstyle={},
pdfborder={0 0 0},
pagebackref,
pdftex]{hyperref}

\renewcommand*{\backref}[1]{}
\renewcommand*{\backrefalt}[4]{
  \ifcase #1 %
   [No citations.]%
  \or
   [#2]%
  \else
   [#2]%
  \fi
}

\newcommand{\calA}{\mathcal{A}}
\newcommand{\calB}{\mathcal{B}}
\newcommand{\calC}{\mathcal{C}}

\newcommand{\calF}{\mathcal{F}}

\newcommand{\calH}{\mathcal{H}}

\newcommand{\calK}{\mathcal{K}}

\newcommand{\calM}{\mathcal{M}}

\newcommand{\calS}{\mathcal{S}}
\newcommand{\calT}{\mathcal{T}}

\newcommand{\calX}{\mathcal{X}}
\newcommand{\calY}{\mathcal{Y}}

\newcommand{\HH}{\mathbb{H}}

\newcommand{\NN}{\mathbb{N}}

\newcommand{\QQ}{\mathbb{Q}}
\newcommand{\RR}{\mathbb{R}}

\newcommand{\ZZ}{\mathbb{Z}}

\renewcommand{\setminus}{{\smallsetminus}}

\newcommand{\st}{\mathbin{\mid}} 
\newcommand{\from}{\colon} 

\newcommand{\nin}{\mathbin{\notin}}

\newcommand{\homeo}{\mathrel{\cong}} 

\newcommand{\isom}{\cong} 

\newcommand{\closure}[1]{{\overline{#1}}}

\newcommand{\bdy}{\partial} 
\newcommand{\dual}{\operatorname{dual}}

\newcommand{\diam}{\operatorname{diam}} 

\newcommand{\MCG}{\mathcal{MCG}} 

\newcommand{\Teich}{{Teichm\"uller~}}

\newcommand{\refsec}[1]{Section~\ref{Sec:#1}}
\newcommand{\refthm}[1]{Theorem~\ref{Thm:#1}}
\newcommand{\refcor}[1]{Corollary~\ref{Cor:#1}}
\newcommand{\reflem}[1]{Lemma~\ref{Lem:#1}}
\newcommand{\refprop}[1]{Proposition~\ref{Prop:#1}}
\newcommand{\refclm}[1]{Claim~\ref{Clm:#1}}
\newcommand{\refrem}[1]{Remark~\ref{Rem:#1}}

\newcommand{\reffig}[1]{Figure~\ref{Fig:#1}}

\theoremstyle{plain}
\numberwithin{equation}{section} 
\newtheorem{theorem}[equation]{Theorem}
\newtheorem{corollary}[equation]{Corollary}
\newtheorem{lemma}[equation]{Lemma}

\newtheorem{proposition}[equation]{Proposition}

\theoremstyle{definition}
\newtheorem{definition}[equation]{Definition}
\newtheorem{remark}[equation]{Remark}
\newtheorem*{remark*}{Remark}
\newtheorem{claim}[equation]{Claim}
\newtheorem*{claim*}{Claim}

\newtheorem*{question*}{Question}
\newtheorem*{answer*}{Answer}
\newtheorem*{case*}{Case}
\newtheorem*{application*}{Application}

\theoremstyle{definition}
\newtheorem{case}{Case}

\newcommand{\fakeenv}{} 

\newenvironment{restate}[2]  
{ 
 \renewcommand{\fakeenv}{#2} 
 \theoremstyle{plain} 
 \newtheorem*{\fakeenv}{#1~\ref{#2}} 
 \begin{\fakeenv}
}
{
 \end{\fakeenv}
}

\newcommand{\carr}{\prec}
\renewcommand{\dual}{\pitchfork} 
\newcommand{\eff}{\dashv} 

\newcommand{\Drift}{{\sf K_0}}
\newcommand{\Wide}{{\sf K_1}}

\newcommand{\CutOff}{{\sf C}}

\newcommand{\Orbit}{{\sf N}}
\newcommand{\Inj}{{\sf N_0}}
\newcommand{\NoMotion}{{\sf N_1}}
\newcommand{\SplitSub}{{\sf N_2}}
\newcommand{\Branch}{{\sf B_0}}

\newcommand{\Reverse}{{\sf R_0}}
\newcommand{\Radius}{{\sf R_1}}

\newcommand{\Volume}{{\sf V}}
\newcommand{\Err}{{\sf E}}

\newcommand{\Induct}{{\sf T_0}}
\newcommand{\Straight}{{\sf T_1}}

\newcommand{\Ind}{{\sf Ind}}
\newcommand{\Long}{{\sf Long}}
\newcommand{\Short}{{\sf Short}}

\newcommand{\ind}{\operatorname{index}}
\newcommand{\card}{\operatorname{card}}

\newcommand{\AC}{\mathcal{AC}}

\begin{document}

\title{On train track splitting sequences}

\author[Masur]{Howard Masur}
\address{\hskip-\parindent
        Department of Mathematics\\
        University of Chicago\\
        Chicago, IL 60607, USA}
\email{masur@math.uchicago.edu}
\thanks{Masur and Mosher are grateful to the National Science Foundation for
  its suport.}

\author[Mosher]{Lee Mosher}
\address{\hskip-\parindent
        Department of Mathematics\\
        Rutgers-Newark, State University of New Jersey\\
        Newark, NJ  07102, USA}
\email{mosher@rutgers.edu}

\author[Schleimer]{Saul Schleimer}
\address{\hskip-\parindent
        Department of Mathematics\\
        University of Warwick\\
        Coventry, CV4 7AL, UK}
\email{s.schleimer@warwick.ac.uk}
\thanks{This work is in the public domain.}

\date{\today}

\begin{abstract}
We show that the subsurface projection of a train track splitting
sequence is an unparameterized quasi-geodesic in the curve complex of
the subsurface.  For the proof we introduce {\em induced tracks}, {\em
efficient position}, and {\em wide curves}.

This result is an important step in the proof that the disk complex is
Gromov hyperbolic.  As another application we show that train track
sliding and splitting sequences give quasi-geodesics in the train
track graph, generalizing a result of Hamenst\"adt [Invent.\ Math.].
\end{abstract}


\maketitle

\section{Introduction}
\label{Sec:Intro}

Train track splitting sequences are an integral part of the study of
surface diffeomorphisms~\cite{Thurston02, PennerHarer92, Mosher03}.
There are links between splitting sequences of tracks and the curves
they carry, on the one hand, and \Teich geodesics and measured
foliations on the other. Of particular importance, in light of the
hierarchy machine~\cite{MasurMinsky00} and the closely related control
of \Teich geodesics~\cite{Rafi05, Rafi07a}, is to understand how a
splitting sequence interacts with subsurface projection.  We give a
detailed structure theorem (\refthm{FellowTravel}) that explains this
interaction.  Our main application is the following result, needed for
Masur and Schleimer's work on the disk complex.

\begin{restate}{Theorem}{Thm:LocalUnparam}
For any surface $S$ with $\xi(S) \geq 1$ there is a constant $Q =
Q(S)$ with the following property: For any sliding and splitting
sequence $\{ \tau_i \}_{i=0}^N$ of birecurrent train tracks in $S$ and
for any essential subsurface $X \subset S$ if $\pi_X(\tau_N) \neq
\emptyset$ then the sequence $\{ \pi_X(\tau_i) \}_{i=0}^N$ is a
$Q$--unparameterized quasi-geodesic in the curve complex $\calC(X)$.
\end{restate}

As another application of \refthm{FellowTravel} we generalize, via a
very different proof, a result of Hamenst\"adt~\cite[Corollary
3]{Hamenstadt09}:

\begin{restate}{Theorem}{Thm:TT}
For any surface $S$ with $\xi(S) \geq 1$ there is a constant $Q =
Q(S)$ with the following property: If $\{ \tau_i \}_{i=0}^N$ is a
sliding and splitting sequence in the train track graph $\calT(S)$,
injective on slide subsequences, then $\{ \tau_i \}$ is a
$Q$--quasi-geodesic.
\end{restate}

For the proof of \refthm{FellowTravel} we introduce {\em induced
tracks}, {\em efficient position}, and {\em wide curves}.  For any
essential subsurface $X \subset S$ and track $\tau \subset S$ there is
an induced track $\tau|X$.  Induced tracks generalize the notion of
subsurface projection of curves.  Efficient position of a curve with
respect to a track $\tau$ is a simultaneous generalization of curves
carried by $\tau$ and curves dual to $\tau$ (called {\em hitting
$\tau$ efficiently} in~\cite{PennerHarer92}).  Efficient position of
$\bdy X$ allows us to pin down the location of the induced track
$\tau|X$.  Wide curves are our combinatorial analogue of curves of
definite modulus in a Riemann surface.  The structure theorem
(\ref{Thm:FellowTravel}) then implies \refthm{LocalUnparam}: this,
together with subsurface projection, controls the motion of a
splitting sequence through the complex of curves $\calC(X)$.

\refthm{TT}, our second application of the structure theorem, is a
direct consequence of \refthm{Quasi}, stated in terms of the marking
graph.  \refthm{Quasi} requires a delicate induction proof
conceptually similar to the hierarchy machine developed
in~\cite{MasurMinsky00}.  We do not deduce \refthm{Quasi} directly
from the results of~\cite{MasurMinsky00}; in particular it is not
known if splitting sequences fellow-travel resolutions of hierarchies.

\subsection*{Acknowledgments} 

We thank Yair Minsky for enlightening conversations. 

\section{Background}
\label{Sec:Back}

We provide the definitions needed for \refthm{FellowTravel} and its
corollaries. 

\subsection{Coarse geometry} 

Suppose that $Q \geq 1$ is a real number.  For real numbers $r, s$ we
write $r \leq_Q s$ if $r \leq Qs + Q$ and say that $r$ is {\em
quasi-bounded} by $s$.  We write $r =_Q s$ if $r \leq_Q s$ and $s
\leq_Q r$; this is called a {\em quasi-equality}.

For a metric space $(\calX, d_\calX)$ and finite diameter subsets $A,
B \subset \calX$ define $d_\calX(A, B) = \diam_\calX(A \cup B)$.
Following Gromov~\cite{Gromov87}, a relation $f \from \calX \to \calY$
of metric spaces is a {\em $Q$--quasi-isometric embedding} if for all
$x, y \in \calX$ we have $d_\calX(x,y) =_Q d_\calY(f(x),f(y))$.  (Here
$f(x) \subset \calY$ is the set of points related to $x$.)  If,
additionally, the $Q$--neighborhood of $f(\calX)$ equals $\calY$ then
$f$ is a {\em $Q$--quasi-isometry} and $\calX$ and $\calY$ are {\em
quasi-isometric}.

If $[m,n]$ is an interval in $\ZZ$ and $f \from [m,n] \to \calY$ is a
quasi-isometric embedding then $f$ is a {\em $Q$--quasi-geodesic}.
Now suppose that $Q > 1$ is a real number, $[m,n]$ and $[p,q]$ are
intervals in $\ZZ$, and $f \from [m,n] \to \calY$ is a relation.  Then
$f$ is a {\em $Q$--unparameterized quasi-geodesic} if there is a
strictly increasing function $\rho \from [p,q] \to [m,n]$ so that $f
\circ \rho$ is a $Q$--quasi-geodesic and for all $i \in [p, q-1]$ the
diameter of $f \left( \big[ \rho(i), \rho(i+1) \big] \right)$ is at
most $Q$.

\subsection{Surfaces, arcs, and curves} 

Let $S = S_{g,n}$ be a compact, connected, orientable surface of genus
$g$ with $n$ boundary components.  The {\em complexity} of $S$ is
$\xi(S) = 3g - 3 + n$.  A {\em curve} in $S$ is an embedding of the
circle into $S$.  An {\em arc} in $S$ is a proper embedding of the
interval $[0,1]$ into $S$.  A curve or arc $\alpha \subset S$ is {\em
trivial} if $\alpha$ separates $S$ and one component of $S \setminus
\alpha$ is a disk; otherwise $\alpha$ is {\em essential}.  A curve
$\alpha$ is {\em peripheral} if $\alpha$ separates $S$ and one
component of $S \setminus \alpha$ is an annulus; otherwise $\alpha$ is
{\em non-peripheral}.  A connected subsurface $X \subset S$ is {\em
essential} if every component of $\bdy X$ is essential in $S$ and $X$
is neither a pair of pants ($S_{0,3}$) nor a peripheral annulus (the
core curve is peripheral).  Note that $X$ inherits an orientation from
$S$.  This, in turn, induces an orientation on $\bdy X$ so that $X$ is
to the left of $\bdy X$.

Define $\calC(S)$ to be the set of isotopy classes of essential,
non-peripheral curves in $S$.  Define $\calA(S)$ to be the set of
proper isotopy classes of essential arcs in $S$.  Let $\AC(S) =
\calC(S) \cup \calA(S)$.  If $\alpha, \beta \in \AC(S)$ then the {\em
geometric intersection number}~\cite{FLP91} of $\alpha$ and $\beta$ is
\[
i(\alpha, \beta) = \min \big\{ | a \cap b | : a \in \alpha, b \in
\beta \big\}.  
\]
A finite subset $\Delta \subset \AC(S)$ is a {\em multicurve} if
$i(\alpha, \beta) = 0$ for all $\alpha, \beta \in \Delta$.  

If $T \subset S$ is a subsurface with $\bdy T$ a union of smooth arcs,
meeting perpendicularly at their endpoints, then define 
\[
\ind(T) = \chi(T) - \frac{c^+(T)}{4} + \frac{c^-(T)}{4}
\]
where $c^\pm(T)$ is the number of outward (inward) corners of $\bdy
T$.  Note that index is additive: $\ind(T \cup T') = \ind(T) +
\ind(T')$ as long as the interiors of $T, T'$ are disjoint.

\subsection{Train tracks} 
\label{Sec:Tracks}

For detailed discussion of train tracks see \cite{Thurston02,
PennerHarer92, Mosher03}.
A {\em pretrack} $\tau \subset S$ is a properly embedded graph in $S$
with additional structure.  The vertices of $\tau$ are called {\em
switches}; every switch $x$ is equipped with a tangent $v_x \in T_x^1
S$.  We require every switch to have valence three; higher valence is
dealt with in~\cite{PennerHarer92}.  The edges of $\tau$ are called
{\em branches}.  All branches are smoothly embedded in $S$.  All
branches incident to a fixed switch $x$ have derivative $\pm v_x$ at
$x$.  

An immersion $\rho \from \RR \to S$ is a {\em train-route} (or simply
a {\em route}) if
\begin{itemize}
\item
$\rho(\RR) \subset \tau$ and 
\item
$\rho(n)$ is a switch if and only if $n \in \ZZ$.
\end{itemize}
The restriction $\rho|[0,\infty)$ is a {\em half-route}.  If $\rho$
factors through $\RR/m\ZZ$ then $\rho$ is a {\em train-loop}.  We
require, for every branch $b$, a train-route travelling along $b$.

For each branch $b$ and point $p \in b$, a component $b'$ of $b
\setminus \{p\}$ is a {\em half-branch}.  Two half-branches $b', b''
\subset b$ are equivalent if $b' \cap b''$ is again a half-branch.
Every switch divides the three incident half-branches into a pair of
{\em small} half-branches on one side and a single {\em large}
half-branch on the other.  A branch $b$ is {\em large} ({\em small})
if both of its half-branches are large (small); if $b$ has one large
and one small half-branch then $b$ is called {\em mixed}.

Let $\calB = \calB(\tau)$ be the set of branches of $\tau$.  A
function $w \from \calB \to \RR_{\geq 0}$ is a {\em transverse
measure} on $\tau$ if $w$ satisfies the {\em switch conditions}: for
every switch $x \in \tau$ we have $w(a) + w(b) = w(c)$, where $a', b'$
are the small half-branches and $c'$ is the large half-branch meeting
$x$.  Let $P(\tau)$ be the projectivization of the cone of transverse
measures; define $V(\tau)$ to be the vertices of the polyhedron
$P(\tau)$.

\begin{figure}[htbp]
$$\begin{array}{c}
\includegraphics[width=10cm]{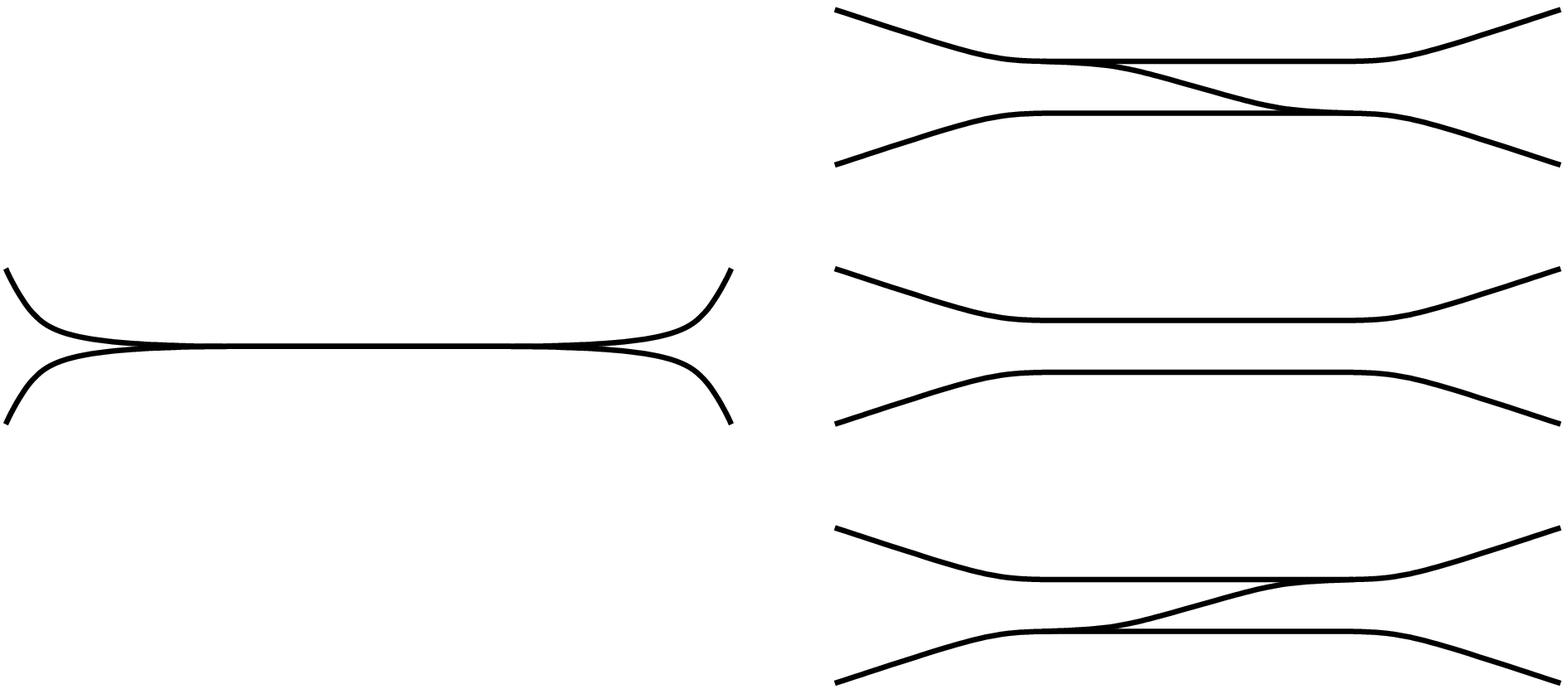} \\
\\
\\
\includegraphics[width=10cm]{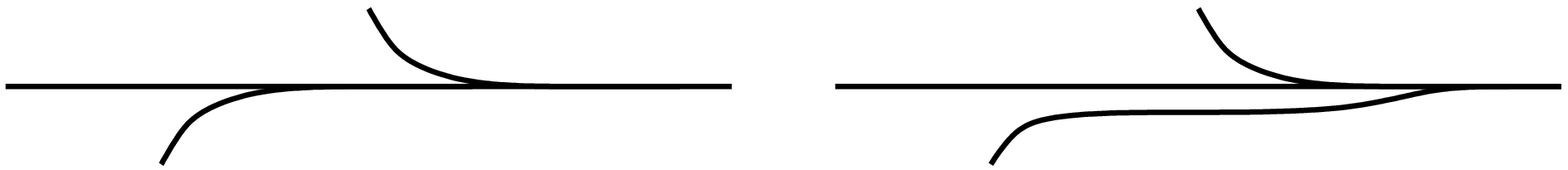}
\end{array}$$
\caption{Top: A large branch admits a right, central, and left
  splitting.  Bottom: a mixed branch admits a slide.}
\label{Fig:SplitSlide}
\end{figure}

We may {\em split} a pretrack along a large branch or {\em slide} it
along a mixed branch; see \reffig{SplitSlide}.  (Slides are called
{\em shifts} in~\cite{PennerHarer92}.)  
The inverse of a split or slide is called a {\em fold}.  Note that the
inverse of a slide may be obtained via a slide followed by an isotopy.

\begin{figure}[htbp]
$$\begin{array}{c}
\includegraphics[height=3.5cm]{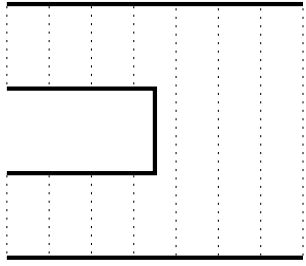}
\end{array}$$
\caption{The local model for $N(\tau)$ near a switch, with horizontal
  and vertical boundary in the correct orientation.  The dotted lines
  are ties.}
\label{Fig:Tie}
\end{figure}

Suppose that $\tau \subset S$ is a pretrack.  Let $N = N(\tau) \subset
S$ be a {\em tie neighborhood} of $\tau$: so $N$ is a union of
rectangles $\{ R_b \st b \in \calB \}$ foliated by vertical intervals
(the {\em ties}). At a switch, the upper and lower thirds of the
vertical side of the large rectangle are identified with the vertical
side of the small rectangles, as shown in \reffig{Tie}.  Since $N$ is
a union of rectangles it follows that $\ind(N) = 0$.  The {\em
horizontal boundary} $\bdy_h N$ is the union of $\bdy_h R_b$, for $b
\in \calB$, while the {\em vertical boundary} is $\bdy_v N =
\closure{\bdy N \setminus \bdy_h N}$.


Let $N = N(\tau)$ be a tie neighborhood.  Let $T$ be a {\em
complementary region} of $\tau$: a component of the closure of $S
\setminus N$.  Define the horizontal and vertical boundary of $T$ to
be $\bdy_h T = \bdy T \cap \bdy_h N$ and $\bdy_v T = \bdy T \cap
\bdy_v N$.  Note that all corners of $T$ are outward, so $\ind(T) =
\chi(T) - \frac{1}{4} |\bdy \bdy_h T|$.

Suppose $\tau \subset S$ is a pretrack.  The subsurface {\em filled}
by $\tau$ is the union of $N$ with all complementary regions $T$ of
$\tau$ that are disks or peripheral annuli.

\begin{definition}
\label{Def:Track}
Suppose that $\tau \subset S$ is a pretrack and $N = N(\tau)$.  We say
that $\tau$ is a {\em train track} if $\tau$ is compact, every
component of $\bdy N$ has at least one corner and every complementary
region $T$ of $\tau$ has negative index.
\end{definition}


\begin{definition}
\label{Def:SplitSeq}
In a {\em sliding and splitting sequence} $\{ \tau_i \}$ of train
tracks each $\tau_{i+1}$ is obtained from $\tau_i$ by a slide or a
split.
\end{definition}

\subsection{Carrying, duality, and efficient position}
\label{Sec:DefEffPos}

Suppose that $\tau \subset S$ is a train track. If $\sigma$ is also a
track, contained in $N = N(\tau)$ and transverse to the ties, then we
write $\sigma \carr \tau$ and say that $\sigma$ is {\em carried} by
$\tau$.  For example, if $\tau$ is a fold of $\sigma$ then $\sigma$ is
carried by $\tau$.

A properly embedded arc or curve $\beta \subset N$ is {\em carried} by
$\tau$ if $\beta$ is transverse to the ties and $\bdy \beta \cap
\bdy_h N = \emptyset$.  Thus if $\beta$ is carried then $\bdy \beta
\subset \bdy_v N$.  Again we write $\beta \carr \tau$ for carried arcs
and curves.

\begin{definition}
\label{Def:EffPos}
Suppose that $\alpha \subset S$ is a properly embedded arc or curve.
Then $\alpha$ is in {\em efficient position} with respect to $\tau$,
denoted $\alpha \eff \tau$, if
\begin{itemize}
\item
every component of $\alpha \cap N$ is a tie or is carried by $\tau$ and
\item
every region $T \subset S \setminus (N \cup \alpha)$ has negative
index or is a rectangle.
\end{itemize}
\end{definition}


Suppose that $\alpha \eff \tau$.  If $\alpha \subset N$ then $\alpha$
is carried, $\alpha \carr \tau$.  If no component of $\alpha \cap N$
is carried then $\alpha$ is {\em dual} to $\tau$ and we write $\alpha
\dual \tau$.  If $\Delta \subset \AC(S)$ is a multicurve then we write
$\Delta \carr \tau$, $\Delta \dual \tau$, or $\Delta \eff \tau$ if all
elements of $\Delta$ are disjointly and simultaneously carried, dual,
or in efficient position.


\begin{remark}
Our notion of duality is called {\em hitting efficiently} by Penner
and Harer~\cite[page 19]{PennerHarer92}. Note that $\alpha \dual \tau$
if and only if $\alpha$ is carried by some extension of the {\em dual
track} $\tau^*$, also defined in~\cite{PennerHarer92}.  Likewise, if
$\alpha \eff \tau$ and $\alpha \cap N$ consists of carried arcs then
$\alpha$ is carried by some extension of $\tau$.
\end{remark}

An index argument proves:

\begin{lemma}
\label{Lem:Essential}
If $\alpha$ is a properly embedded curve or arc in efficient position
with respect to a train track $\tau \subset S$ then $\alpha$ is
essential and non-peripheral in $S$.  \qed
\end{lemma}


\noindent
One of the goals of this paper is to prove the converse of
\reflem{Essential}; this is done in \refthm{EffPos}.  Following
\reflem{Essential} we may define $\calC(\tau) = \{ \alpha \st \alpha
\carr \tau \}$ and $\calC^*(\tau) = \{ \alpha \st \alpha \dual \tau
\}$.  Notice that if $\sigma \carr \tau$ is a track then
$\calC(\sigma) \subset \calC(\tau)$ and $\calC^*(\tau) \subset
\calC^*(\sigma)$.

A branch $b \in \calB(\tau)$ is {\em recurrent} if there is some
$\alpha \carr \tau$ that meets $R_b$.  The track $\tau$ is {\em
recurrent} if every branch is recurrent.  
{\em Transverse recurrence} is defined by replacing carrying by
duality~\cite[page 20]{PennerHarer92}.  
The track $\tau$ is {\em birecurrent} if $\tau$ is recurrent and
transversely recurrent~\cite[Section 1.3]{PennerHarer92}.  In a slight
departure from Penner and Harer's terminology~\cite[page
27]{PennerHarer92} we will call a birecurrent track $\tau$ {\em
complete} if all complementary regions have index $-1/2$.  (When $S =
S_{1,1}$ there is, instead, a single complementary region with index
$-1$.)


\begin{lemma}
\label{Lem:ManyDuals}
Suppose that $\sigma \subset S$ is a birecurrent track.  Then
$\calC^*(\sigma)$ has infinite diameter inside of $\calC(S)$.
\end{lemma}

\begin{proof}
Let $\tau$ be a complete track extending $\sigma$~\cite[Corollary
  1.4.2]{PennerHarer92}.  Section 3.4 of~\cite{PennerHarer92} and a
dimension count gives a lamination $\lambda \dual \tau$ so that
$i(\lambda, \alpha) \neq 0$ for all $\alpha \in \calC(S)$.  Now an
argument of Kobayashi~\cite{Kobayashi88b}, refined by Luo~\cite[page
  124]{MasurMinsky99}, implies that $\calC^*(\tau) \subset
\calC^*(\sigma)$ has infinite diameter.
\end{proof}

\subsection{Vertex cycles}

When $\alpha \carr \tau$ is a curve there is a transverse measure
$w_\alpha$ defined by taking $w_\alpha(b) = |\alpha \cap t|$ where $t$
is any tie of the rectangle $R_b$.  Conversely, for any integral
transverse measure $w$ there is a multicurve $\alpha_w$ --- take
$w(b)$--many horizontal arcs in $R_b$ and glue endpoints as dictated
by the switch conditions.

Note that if $v \in V(\tau)$ then there is a minimal integral measure
$w$ projecting to $v$.  Since $v$ is an extreme point of $P(\tau)$
deduce that $\alpha_w$ is an embedded curve.  We call $\alpha_w$ a
{\em vertex cycle} of $\tau$ and henceforth use $V(\tau)$ to denote
the set of vertex cycles.

\subsection{Wide curves}

Let $N = N(\tau)$ be a tie neighborhood. 

\begin{definition}
\label{Def:Wide}
A multicurve $\Delta \eff \tau$ is {\em wide} if there is an
orientation of the components of $\Delta$ so that
\begin{itemize}
\item
for every $b \in \calB(\tau)$, all arcs of $\Delta \cap R_b$ are to
the right of each other (see the top of \reffig{Right}) and
\item
for every complementary region $T$ of $\tau$, all arcs of $\Delta \cap
T$ are to the right of each other (see the bottom of \reffig{Right}).
\end{itemize}
\end{definition}


\begin{figure}[htbp]
$$
\begin{array}{c} 
\includegraphics[height=2.2cm]{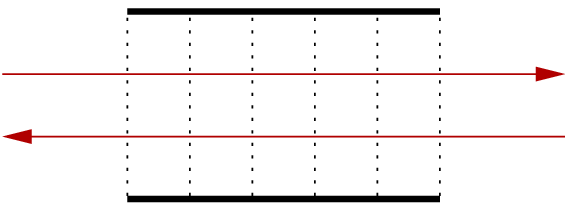} \\
\\
\includegraphics[height=3.5cm]{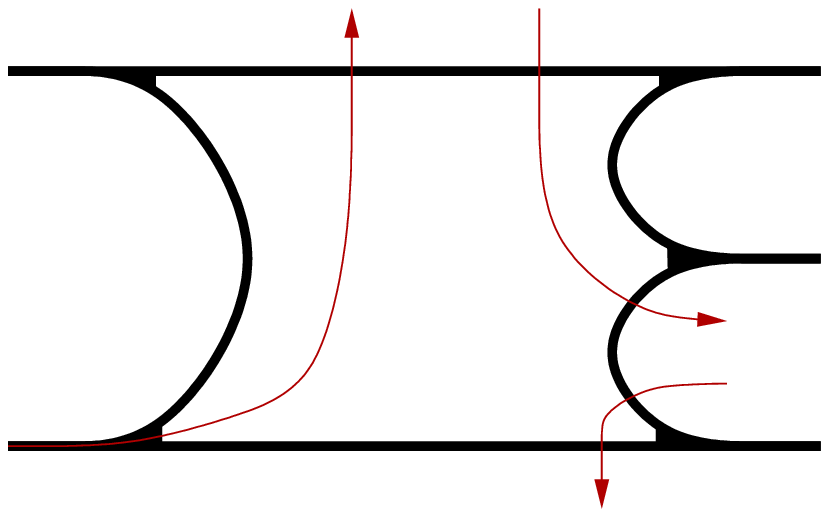}
\end{array}
$$
\caption{Above: Arcs of $\Delta \cap R_b$ to the right of each other.
  The vertical dotted line are ties; the heavy horizontal lines are
  arcs of $\bdy_h N$.  Below: Arcs meeting the complementary region
  $T$, all to the right of each other.}
\label{Fig:Right}
\end{figure}

It follows from the definition that if $\Delta \eff \tau$ is wide then
for any branch $b \in \calB(\tau)$ the intersection $\Delta \cap R_b$
has at most two components. 


\begin{lemma}
\label{Lem:VerticesWide}
Every vertex cycle $\alpha \in V(\tau)$ is wide.
\end{lemma}


For an even more precise characterization of vertex cycles see Lemma
3.11.3 of~\cite{Mosher03}.
The proof below recalls a surgery technique used in the sequel.

\begin{proof}[Proof of \reflem{VerticesWide}]
We prove the contrapositive.  Suppose that $\alpha$ is not wide.
Orient $\alpha$.  There are three cases.

Suppose there is a branch $b \subset \tau$ and an oriented tie $t
\subset R_b$ where $x$ and $y$ are consecutive (along $t$) points of
$\alpha \cap t$ so that the signs of intersection at $x$ and $y$ are
equal.  Let $[x,y]$ be the subarc of $t$ bounded by $x$ and $y$.
Surger $\alpha$ along $[x,y]$ to form curves $\beta, \gamma \carr
\tau$.  See \reffig{Surger}.  Thus $w_\alpha = w_\beta + w_\gamma$ and
$\alpha$ is not a vertex cycle.

\begin{figure}[htbp]
\labellist
\small\hair 2pt
\pinlabel {$t$} [l] at 149.1 173.6
\pinlabel {$\alpha$} [l] at 291.6 145.5
\pinlabel {$\beta$} [l] at 291.6 127
\pinlabel {$\gamma$} [l] at 291.6 37.1
\endlabellist
$$\begin{array}{c}
\includegraphics[width=5.5cm]{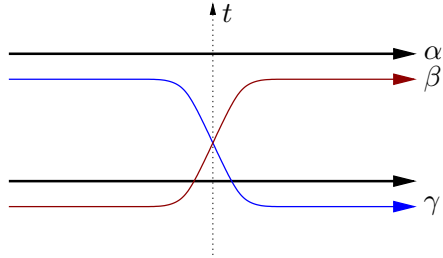} 
\end{array}$$
\caption{Surgery when adjacent intersections have the same sign.}
\label{Fig:Surger}
\end{figure}

Suppose instead that $x, y, z$ are consecutive (along $t$) points of
$\alpha \cap t$ with alternating sign.  In this case there is again a
surgery along $[x,z]$ producing curves $\beta$ and $\gamma$.  See
\reffig{Surger2} for one of the possible arrangements of $\alpha$,
$\beta$ and $\gamma$.  Again $w_\alpha = w_\beta + w_\gamma$ is a
non-trivial sum and $\alpha$ is not a vertex cycle.

\begin{figure}[htbp]
\labellist
\small\hair 2pt
\pinlabel {$t$} [l] at 149.1 264.0
\pinlabel {$\alpha$} [l] at 291.6 235.9
\pinlabel {$\beta$} [l] at 291.6 217.3
\pinlabel {$\gamma$} [l] at 291.6 163.6
\endlabellist
$$\begin{array}{c}
\includegraphics[width=5.5cm]{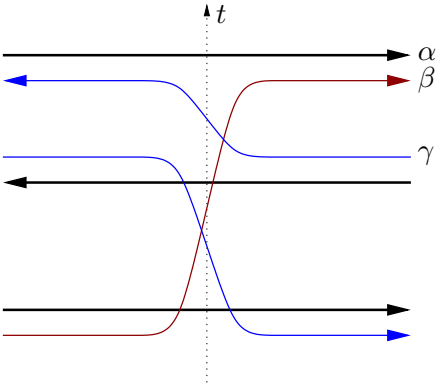} 
\end{array}$$
\caption{Surgery when the three intersections have alternating sign.}
\label{Fig:Surger2}
\end{figure}

In the remaining case $w_\alpha(a) \leq 2$ for all $a \in \calB$ and
there are branches $b,c \in \calB$ where the arcs of $\alpha \cap R_b$
are to the right of each other while the arcs of $\alpha \cap R_c$ are
to the left of each other.  See \reffig{Disagree} for the two ways
$\alpha$ may be carried by $\tau$.

\begin{figure}[htbp]
$$\begin{array}{c}
\includegraphics[height=2.5cm]{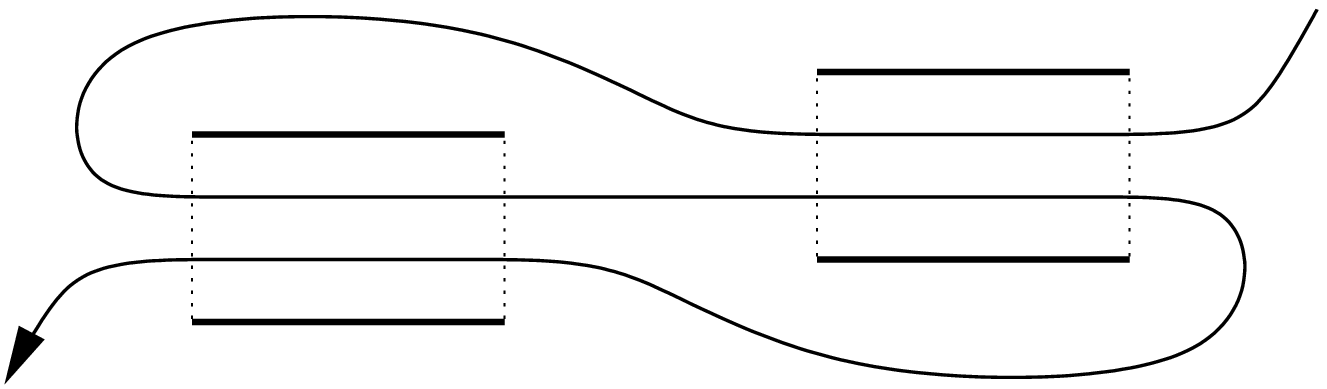} \\
\includegraphics[height=2.5cm]{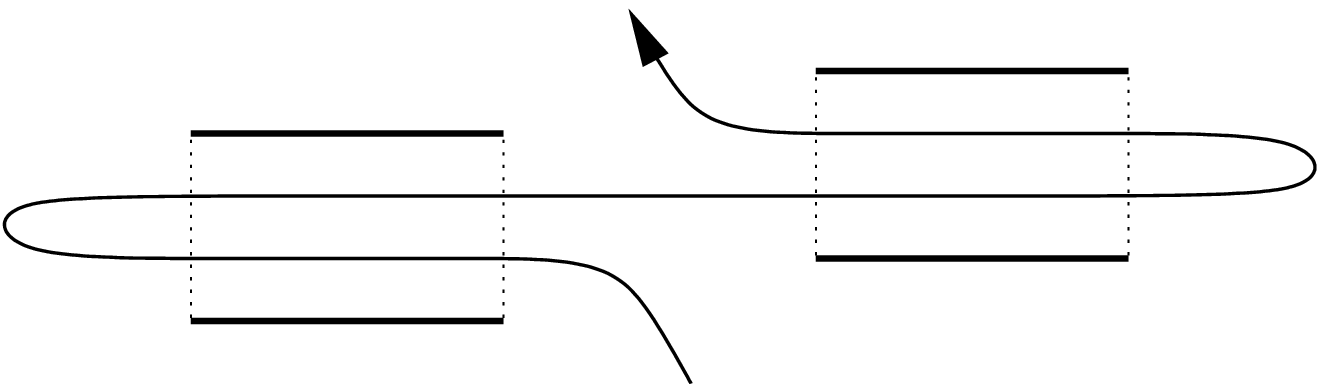} 
\end{array}$$
\caption{The curve $\alpha$ meets $R_b$ and $R_c$ twice.}
\label{Fig:Disagree}
\end{figure}

If $\alpha$ is carried as in the top line of \reffig{Disagree} then
surger $\alpha$ in both rectangles $R_b$ and $R_c$ as was done in
\reffig{Surger}.  This shows that $w_\alpha$ is a non-trivial sum and
so $\alpha$ is not a vertex cycle.  Now suppose that $\alpha$ is
carried as in the bottom line of \reffig{Disagree}.  Note that the
closure of $\alpha \setminus (R_b \cup R_c)$ is a union of four arcs.
Two of these, $\beta'$ and $\gamma'$, meet both $R_b$ and $R_c$.
Since $S$ is orientable no tie-preserving isotopy of $N$ throws
$\beta'$ onto $\gamma'$.
Let $\alpha' \cup \alpha'' = \alpha \setminus (\beta'
\cup \gamma')$.  Create an embedded curve $\beta$ by taking two
parallel copies of $\beta'$ and joining them to $\alpha'$ and
$\alpha''$.  Similarly create $\gamma$ by joining two parallel copies
of $\gamma'$ to the arcs $\alpha'$ and $\alpha''$.  It follows that
$w_\beta \neq w_\gamma$.  Since $2w_\alpha = w_\beta + w_\gamma$ again
$\alpha$ is not a vertex cycle.
\end{proof}

\subsection{Combing}

Suppose that $\alpha \carr \tau$ has $w_\alpha(b) \leq 1$ for every
branch $b \subset \tau$.  Orient and transversely orient $\alpha$ to
agree with the orientation of the surface $S$.  We think of the
orientation as pointing in the $x$--direction and the transverse
orientation pointing in the $y$--direction.  A half-branch $b \subset
\tau \setminus \alpha$, sharing a switch with $\alpha$, {\em twists to
the right} if any train-route through $b$ locally has positive slope.
Otherwise $b$ twists to the left.  If all branches on one side of
$\alpha$ twist to the right then that side of $\alpha$ has a {\em
right combing}, and similarly for a left combing.  See
\reffig{AnnSwap} for an example where both sides are combed to the
left.

\subsection{Curve complexes and subsurface projection}

For more information on the curve complex see~\cite{MasurMinsky99,
MasurMinsky00}.  Impose a simplicial structure on $\AC(S)$ where
$\Delta \subset \AC(S)$ is a simplex if and only if $\Delta$ is a
multicurve.  The complex of curves $\calC(S)$ and the arc complex
$\calA(S)$ are the subcomplexes spanned by curves and arcs,
respectively.  Note that if the complexity $\xi(S)$ is at least two
then $\calC(S)$ is connected~\cite[Proposition~2]{Harvey81}.  For
surfaces of lower complexity we alter the simplicial structure on
$\calC(S)$.  

Define the Farey tessellation $\calF$ to have vertex set $\QQ \cup
\{\infty\}$.  A collection of {\em slopes} $\Delta \subset \calF$
spans a simplex if $ps - rq = \pm 1$ for all $p/q, r/s \in \Delta$.
If $S = S_{1,1}$ or $S_{0,4}$ we take $\calC(S) = \calF$; that is,
there is an edge between curves that intersect in exactly one point
(for $S_{1,1}$) or two points (for $S_{0,4}$).  Note that for surfaces
with $\xi(S) \geq 1$ the inclusion of $\calC(S)$ into $\AC(S)$ is a
quasi-isometry.

Suppose now that $X \homeo S_{0,2}$ is an annulus.  Define $\calA(X)$
to be the set of all essential arcs in $X$, up to isotopy fixing $\bdy
X$ pointwise.  For $\alpha, \beta \in \calA(X)$ define
\[ 
i(\alpha, \beta) = \min \big\{ | (a \cap b) \setminus \bdy X | 
                                    : a \in \alpha, b \in \beta \big\}.
\]
As usual, multicurves give simplices for $\calA(X)$. 

If $\alpha, \beta$ are vertices of $\calC(S)$, $\calA(S)$, or $\AC(S)$
then define $d_S(\alpha, \beta)$ to be the minimal number of edges in
a path, in the one-skeleton, connecting $\alpha$ to $\beta$; the
containing complex will be clear from context.  Note that if $\alpha$
$\beta$ are distinct arcs of $\calA(X)$, when $X$ is an annulus, then
$d_X(\alpha, \beta) = 1 + i(\alpha,
\beta)$~\cite[Equation~2.3]{MasurMinsky00}.

As usual, suppose that $\xi(S) \geq 1$.  Fix an essential subsurface
$X \subset S$ with $\xi(X) < \xi(S)$.  We suppose that $X$ is either a
non-peripheral annulus or a surface of complexity at least one.  (The
case of an essential annulus inside of $S_1$ is not relevant here.)
Following~\cite{MasurMinsky00}, we will define the {\em subsurface
projection} relation $\pi_X \from \AC(S) \to \AC(X)$.  Let $S^X$ be
the cover of $S$ corresponding to the inclusion $\pi_1(X) < \pi_1(S)$.
The surface $S^X$ is not compact; however, there is a canonical (up to
isotopy) homeomorphism between $X$ and the Gromov compactification of
$S^X$.  This identifies the arc and curve complexes of $X$ and $S^X$.
Fix $\alpha \in \AC(S)$.  Let $\alpha^X$ be the preimage of $\alpha$
in $S^X$.

If $\alpha^X$ contains a non-peripheral curve in $S^X$ then
$\pi_X(\alpha) = \{ \alpha \}$.  Otherwise, place every essential arc
of $\alpha^X$ into the set $\pi_X(\alpha)$.  If neither obtains then
$\pi_X(\alpha) = \emptyset$.  If $\pi_X(\alpha) = \emptyset$ then we
say that $\alpha$ {\em misses} $X$.  If $\pi_X(\alpha) \neq \emptyset$
then $\alpha$ {\em cuts} $X$.



Suppose that $\alpha, \beta \in \calC(S)$.  If $\pi_X(\alpha)$ and
$\pi_X(\beta)$ are nonempty define
\[
d_X(\alpha, \beta) = \diam_X \big( \pi_X(\alpha) \cup \pi_X(\beta)
\big).
\]
Likewise define the distance $d_X(A, B)$ between finite sets $A, B
\subset \calC(S)$.  When $\tau$ is a track, we use the shorthand
$\pi_X(\tau)$ for the set $\pi_X(V(\tau))$.  If $\sigma$ is also a
track, we write $d_X(\tau, \sigma)$ for the distance $d_X(\pi_X(\tau),
\pi_X(\sigma))$.

We end with a lemma connecting the subsurface projection of carried
(or dual) curves to the behavior of wide curves. 

\begin{lemma}
\label{Lem:Decompose}
Suppose that $X \subset S$ is an essential surface and $\tau$ is a
track.  If some $\alpha \carr \tau$ ($\alpha \dual \tau$) cuts $X$
then there is a vertex cycle $\beta \carr \tau$ (wide dual $\beta
\dual \tau$) cutting $X$. 
\end{lemma}

\begin{proof}
Some multiple of $\alpha \carr \tau$ is a sum of vertices: $m \cdot
w_\alpha = \sum n_i w_i$ where $w_i$ is the integral transverse
measure for the vertex $\beta_i \in V(\tau)$.  Via a sequence of
tie-preserving isotopies of $N$ we may arrange for all of the
$\beta_i$ to realize their geometric intersection with each
other. Note that there is an isotopy representative of $\alpha$
contained inside of a small neighborhood of the union $B = \cup
\beta_i$.

To prove the contrapositive, suppose that none of the $\beta_i$ cut
$X$.  It follows that $X$ may be isotoped in $S$ to be disjoint from
$B$.  Thus $\alpha$ misses $X$, as desired.  A similar
discussion applies when $\alpha \dual \tau$. 
\end{proof}

\section{Induced tracks}
\label{Sec:Induced}

Suppose that $\tau \subset S$ is a train track.  Suppose that $X
\subset S$ is an essential subsurface with $\xi(X) < \xi(S)$.  Let
$S^X$ be the corresponding cover of $S$.  Let $\tau^X$ be the preimage
of $\tau$ in $S^X$; note that the pretrack $\tau^X$ satisfies all of
the axioms of a train track except compactness.

Define $\AC(\tau^X)$ to be the set of essential arcs and essential,
non-peripheral curves properly embedded in the Gromov compactification
of $S^X$ with interior a train-route or train-loop carried by
$\tau^X$.  A bit of caution is required here --- inessential arcs and
peripheral curves may be carried by $\tau^X$ but these are not
admitted into $\AC(\tau^X)$.  Define $\calA(\tau^X), \calC(\tau^X)
\subset \AC(\tau^X)$ to be the subsets of arcs and curves
respectively.  Define $\AC^*(\tau^X)$ to be the set of dual essential
arcs and dual essential, non-peripheral curves, up to isotopies fixing
$\tau^X$ setwise. 

\subsection{Induced tracks for non-annuli}
\label{Sec:InducedForNonAnnuli}

If $X$ is not an annulus define $\tau|X$, the {\em induced track}, to
be the union of the branches of $\tau^X$ crossed by an element of
$\calC(\tau^X)$.  

\begin{lemma}
\label{Lem:Compact}
If $X$ is not an annulus then the induced track $\tau|X$ is compact.
\end{lemma}

\begin{proof}
Note train-routes in $\tau^X$ that are mapped properly to $S^X$ are
uniform quasi-geodesics in $S^X$~\cite[Proposition 3.3.3]{Mosher03}.
Thus there is a compact core $X' \subset S^X$, homeomorphic to $X$, so
that any route meeting $S^X \setminus X'$ has one endpoint on the
Gromov boundary of $S^X$.  It follows that $\tau|X \subset X'$.
\end{proof}


Note that $\tau|X$ may not be a train track: $N = N(\tau|X)$ may have
smooth boundary components and complementary regions with non-negative
index.
However, since all complementary regions of $\tau^X$ have negative
index it follows that if a complementary region $T$ of $\tau|X$ has
non-negative index then $T$ is a peripheral annulus meeting a smooth
component of $\bdy N$.

The definition of $\tau|X$ implies that $\tau|X$ is recurrent.
Carrying, duality, efficient position and wideness with respect to an
induced track are defined as in \refsec{DefEffPos}.  Define
$\calC(\tau|X) \subset \calC(X)$, the subset of curves carried by
$\tau|X$.  Note that $\calC(\tau|X) = \calC(\tau^X)$.  Define
$\AC^*(\tau|X) \subset \AC(X)$ to be the subset of arcs and curves
dual to $\tau|X$.  Note that $\AC^*(\tau|X) \supset \AC^*(\tau^X)$.

Now, $\tau|X$ fails to be transversely recurrent exactly when it
carries a peripheral curve.  We say that a branch $b \subset \tau|X$
is {\em transversely recurrent with respect to arcs and curves} if
there is $\alpha \in \AC^*(\tau|X)$ meeting $b$.  Then $\tau|X$ is
{\em transversely recurrent with respect to arcs and curves} if every
branch $b$ is.



\begin{lemma}
\label{Lem:TransverseRecurrenceAC}
Suppose that $\tau$ is transversely recurrent in $S$.  Then $\tau|X$
is transversely recurrent with respect to arcs and curves in $X$.
Furthermore: suppose that $\tau|X$ is transversely recurrent with
respect to arcs and curves in $X$.  If $\sigma \subset \tau|X$ is a
train track then $\sigma$ is transversely recurrent in $X$. 
\end{lemma}

\begin{proof}
The first claim follows from the definitions.  An index argument
proves the second claim.
\end{proof}





Here is our second surgery argument.

\begin{lemma}
\label{Lem:Surgery}
Suppose that $\tau$ is a track and $X \subset S$ is an essential
subsurface, yet not an annulus.  For every $\alpha \in \calA(\tau^X)$
at least one of the following holds:
\begin{itemize}
\item
There is an arc $\beta \in \calA(\tau^X)$ so that $\beta$ is wide and
$i(\alpha, \beta) = 0$.
\item
There is a curve $\gamma \in \calC(\tau|X)$ so that $i(\alpha, \gamma)
\leq 2$.
\end{itemize}
The statement also holds replacing $\calA, \calC$ by $\calA^*,
\calC^*$. 
\end{lemma}

\begin{proof}
The proof is modelled on that of \reflem{VerticesWide}.  If $\alpha
\carr \tau^X$ is wide we are done.  If not, as $\alpha$ is a
quasi-geodesic~\cite[Proposition 3.3.3]{Mosher03}, 
orient $\alpha$ so that $\alpha$ is wide outside of a compact core for
$S^X$.  Now we induct on the total number of arcs of intersection
between $\alpha$ and rectangles $R_b \subset N(\tau^X)$ meeting the
compact core.

Let $t$ be a tie of $R_b$.  Orient $t$. Suppose that $x, y$ are
consecutive (along $t$) points of $\alpha \cap t$.  Suppose that the
sign of intersection at $x$ equals the sign at $y$.  Let $[x,y]$ be
the subarc of $t$ bounded by $x$ and $y$.  As in \reflem{VerticesWide}
surger $\alpha$ along $[x,y]$ to form an arc $\beta'$ and a curve
$\gamma$.  See \reffig{Surger} with $\beta'$ substituted for $\beta$.

Note that $\gamma$ is essential in $S^X$, by an index argument.  If
$\gamma$ is non-peripheral then the second conclusion holds.  So
suppose that $\gamma$ is peripheral.  Then $\alpha$ is obtained by
Dehn twisting $\beta'$ about $\gamma$.  So $\beta'$ is properly
isotopic to $\alpha$ and has smaller intersection with $R_b$; thus we
are done by induction.

Suppose instead that $x, y, z$ are consecutive (along $t$) points of
$\alpha \cap t$, with alternating sign.  Surger $\alpha$ along $[x,z]$
to form an arc $\beta'$ and a curve $\gamma$.  See \reffig{Surger2},
with $\beta'$ substituted for $\beta$, for one of the possible
arrangements of $\alpha$, $\beta'$, and $\gamma$.  Again $\gamma$ is
essential.  If $\gamma$ is non-peripheral then the second conclusion
holds and we are done.  If $\gamma$ is peripheral then, as $\alpha$
and $\beta'$ differ by a half-twist about $\gamma$, we find that
$\beta'$ is properly isotopic to $\alpha$.  Since $\beta'$ has smaller
intersection with $R_b$ we are done by induction.

All that remains is the case that $\alpha$ meets every rectangle $R_b$
in most a pair of arcs of opposite orientation. For every branch $b$
where $\alpha$ meets $R_b$ twice, choose a subarc $t_b$ of a tie in
$R_b$ so that $\alpha \cap t_b = \bdy t_b$.  We call $t_b$ a {\em
chord} for $\alpha$.  For every $t_b$ there is a subarc $\alpha_b
\subset \alpha$ so that $\bdy t_b = \bdy \alpha_b$.  A chord $t_b$ is
{\em innermost} if there is no chord $t_c$ with $\alpha_c$ strictly
contained in $\alpha_b$.  Let $t_b$ be the first innermost chord.  Let
$\alpha'$ be the component of $\alpha \setminus \alpha_b$ before
$\alpha_b$.  Build a route $\beta$ by taking two copies of $\alpha'$
and joining them to $\alpha_b$.  Note that $\beta \cap R_c$ is a
single arc or a pair of arcs exactly as $\alpha_b$ or $\alpha'$ meets
$R_c$.  Thus $\beta$ is wide.  Also, $\beta$ is essential: otherwise
$t_b \cup \alpha_b$ bounds a disk with index one-half, a
contradiction.  By construction $i(\alpha, \beta) = 0$ and
\reflem{Surgery} is proved.
\end{proof}

\subsection{Induced tracks for annuli}
\label{Sec:InducedForAnnuli}

Suppose that $X \subset S$ is an annulus.  Define $\tau|X$ to be the
union of branches $b \subset \tau^X$ so that some element of
$\calA(\tau^X)$ travels along $b$.  (Note that $\tau|X$, if nonempty,
is not compact.)  Define $\calA(\tau|X) = \calA(\tau^X)$ and also the
duals $\calA^*(\tau|X) \supset \calA^*(\tau^X)$.

Define $V(\tau|X)$ in $\calA(\tau|X)$ to be the set of wide carried
arcs.  Define $V^*(\tau|X)$ dually.  


\begin{lemma}
\label{Lem:WideInAnnuli}
Suppose that $X \subset S$ is an essential annulus.  If
$\calA^{(*)}(\tau|X)$ is nonempty then $V^{(*)}(\tau|X)$ is nonempty.
Let $N = N(\tau|X)$.  If $\gamma \eff \tau|X$ is a wide essential arc
then $\gamma$ meets each rectangle of $N$ and each region of $S^X
\setminus N$ in at most a single arc.
\end{lemma}

For example, if $\gamma \carr \tau|X$ is a wide essential arc then
$\gamma$ embeds into $\tau|X$.

\begin{proof}[Proof of \reflem{WideInAnnuli}]
We prove the second conclusion; the first is similar.  Suppose that
$R$ is either a rectangle or region so that $\gamma \cap R$ is a pair
of arcs to the right of each other.  Let $\delta$ be an arc properly
embedded in $R \setminus \gamma$ so that $\delta \cap \gamma = \bdy
\delta$.  Let $\gamma'$ be the component of $\gamma \setminus \bdy
\delta$ so that $\bdy \gamma' = \bdy \delta$.  If $\gamma' \cup
\delta$ bounds a disk in $S^X$ then this disk has index one-half and
we contradict efficient position.  If $\gamma' \cup \delta$ bounds an
annulus then $\gamma$ was not essential, another contradiction.
\end{proof}

Suppose that $\alpha$ is the core curve of the annulus $X$.

\begin{lemma}
\label{Lem:AnnulusEquality}
If $\alpha$ is not carried by $\tau|X$ then $V(\tau|X) =
\calA(\tau|X)$.  If $\alpha$ is not dual to $\tau|X$ then $V^*(\tau|X)
= \calA^*(\tau|X)$. \qed
\end{lemma}



\begin{lemma}
\label{Lem:Combed}
Suppose that $\alpha \carr \tau|X$.  One side of $\alpha$ is combed if
and only if both sides are combed in the same direction if and only if
some isotopy representative of $\alpha$ is dual to $\tau|X$. \qed
\end{lemma}


\section{Finding efficient position}
\label{Sec:FindEffPos}


After discussing the various sources of non-uniqueness we prove in
\refthm{EffPos} that efficient position exists

Let $N = N(\tau)$; suppose that $\alpha \eff \tau$.  A rectangle $T
\subset S \setminus (N \cup \alpha)$ is {\em vertical} if $\bdy T$ has
a pair of opposite sides meeting $\alpha$ and $\bdy_v N$ respectively.
Define {\em horizontal} rectangles similarly.  \reffig{RectSwap}
depicts the two kinds of {\em rectangle swap}.

\begin{figure}[htbp]
$$\begin{array}{c}
\includegraphics[height=3.5cm]{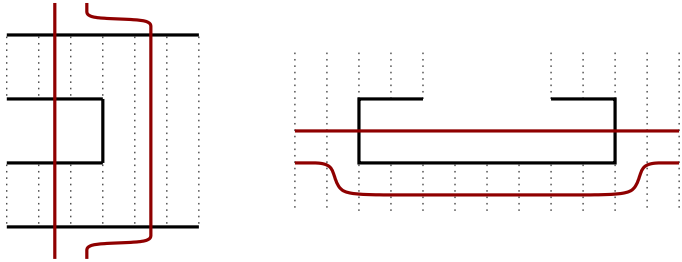} 
\end{array}$$
\caption{Pieces of $N$ are shown, with vertical and horizontal
  boundary in the correct orientation; the dotted lines are ties.  The
  left and right pictures show a vertical and horizontal rectangle
  swap, respectively.}
\label{Fig:RectSwap}
\end{figure}

Now suppose that $\alpha \carr \tau$, every rectangle $R_b \subset N$
meets $\alpha$ in at most a single arc, and one side of $\alpha$ is
combed.  Let $A$ be a small regular neighborhood of $\alpha$.  Then an
{\em annulus swap} interchanges $\alpha$ and the component of $\bdy A$
on the combed side.  See \reffig{AnnSwap}.

\begin{figure}[htbp]
$$\begin{array}{c}
\includegraphics[height=3.5cm]{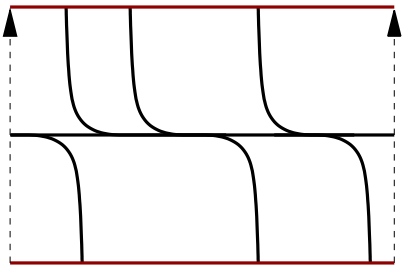} 
\end{array}$$
\caption{Both sides of $\alpha$ are combed (to the left).  Thus both
  boundary components of the annulus shown are dual to $\tau$ and both
  differ from the carried core curve by an annulus swap.}
\label{Fig:AnnSwap}
\end{figure}

\begin{theorem}
\label{Thm:EffPos}
Suppose that $\xi(S) \geq 1$ and $\tau \subset S$ is a birecurrent
train track.  Suppose that $\Delta \subset \AC(S)$ is a multicurve.
Then efficient position for $\Delta$ with respect to $\tau$ exists and
is unique up to rectangle swaps, annulus swaps, and isotopies of $S$
preserving the foliation of $N(\tau)$ by ties.
\end{theorem}

\begin{remark}
When $S = S_1$ is a torus, Lemma~14 of~\cite{Gueritaud09} proves the
existence of efficient position for curves with respect to Reebless
bigon tracks.  Uniqueness of efficient position follows from a slight
generalization of \refsec{Unique} using {\em bigon swaps}.
\end{remark}

\subsection{Uniqueness of efficient position}
\label{Sec:Unique}



Suppose that $\alpha$ and $\beta$ are isotopic curves and in efficient
position with respect to $\tau$.  We induct on $i(\alpha, \beta)$.
For the base case suppose that $|\alpha \cap \beta| = 0$.  Then
$\alpha$ and $\beta$ cobound an annulus $A \subset S$ so that $\bdy A$
has no corners~\cite[Lemma~2.4]{Epstein66}.
Since $N = N(\tau)$ is a union of rectangles the intersection $A \cap
N$ is also a union of rectangles.  Thus $\ind(N \cap A) = 0$.  By the
hypothesis of efficient position any region $T \subset \closure{A
\setminus N}$ has non-positive index and has all corners outwards.  By
the additivity of index it follows that $\ind(T) = 0$.  It follows
that each region $T$ is either an annulus without corners or a
rectangle.

Suppose that some region $T$ is an annulus without corners.  Then we
must have $T = A$.  For if $\bdy T$ meets $\bdy N$ then $\bdy N$ has a
component without corners, contrary to assumption.  Since $T = A$ it
follows that $\alpha$ and $\beta$ are isotopic in the complement of
$N$ and we are done.

So we may assume that all regions of $A \setminus N$ are rectangles.
(In particular, $A \cap N \neq \emptyset$.)  Note that if a region $R$
is a horizontal rectangle then there is no obstruction to doing a
rectangle swap across $R$.  After doing all such swaps we may assume
that $A \setminus N$ contains no horizontal rectangles.

We now abuse terminology slightly by assuming that the position of $N$
determines that of $\tau$.  So if $A$ contains vertical rectangles
then there are switches of $\tau$ contained in $A$.  This implies that
$A$ contains half-branches of $\tau$.  Let $b'$ be a half-branch in
$A$, meeting $\bdy A$.  If $b'$ is large then there is a vertical
rectangle swap removing three half-branches from $A$.  After doing all
such swaps we may assume that any such $b'$ is small.  If $R$ is a
vertical rectangle meeting $\bdy A$ then $R$ has two horizontal sides.
If neither of these meets a switch on its interior then again there is
a swap removing three half-branches from $A$.

After doing all such swaps if there are still vertical rectangles in
$A$ then we proceed as follows: every vertical rectangle must have a
horizontal side that properly contains the horizontal side of another
vertical rectangle.  (For example, in \reffig{AnnSwap} number the
rectangles above the core curve $R_0, R_1, R_2$ from left to right.
Note that the left horizontal side of $R_i$ strictly contains the
right horizontal side of $R_{i-1}$.)  It follows that the union of
these vertical rectangles gives an annulus swap which we perform.
Thus, we are reduced to the situation where $A$ contains no horizontal
or vertical rectangles.

If $A \subset N$ then $\alpha$ and $\beta$ are both carried.  For any
tie $t \subset N$, any component $t' \subset t \cap A$ is an essential
arc in $A$.  (To see this, suppose that $t'$ is inessential.  Let $B
\subset A$ be the bigon cobounded by $t'$ and $\alpha' \subset
\alpha$, say.  Since $\alpha$ is carried, $\alpha'$ is transverse to
the ties. We define a continuous involution on $\alpha'$; for every
tie $s$ and for every component $s' \subset s \cap B$ transpose the
endpoints of $s'$.  As this involution is fixed point free, we have
reached a contradiction.)  It follows that $A$ is foliated by subarcs
of ties and we are done.

There is one remaining possibility in the base case of our induction:
$A \cap N \neq \emptyset$, $A \not\subset N$, and $A$ contains no
switches of $\tau$.  Thus every region of $A \cap N$ and of $A
\setminus N$ is a rectangle meeting both $\alpha$ and $\beta$.  Any
region $R$ of $A \cap N$ is foliated by (subarcs of) ties and, as
above, all ties meet $R$ essentially.  Thus $R$ gives a parallelism
between (carried arcs) ties of $\alpha$ and $\beta$.  It follows that
$A$ gives an isotopy between $\alpha$ and $\beta$, sending ties to
ties. This completes the proof of uniqueness when $|\alpha \cap \beta|
= 0$.

For the induction step assume $|\alpha \cap \beta| > 0$.  Since
$\alpha$ is isotopic to $\beta$ the Bigon
Criterion~\cite[Lemma~2.5]{Epstein66},
\cite[Proposition~1.3]{FarbMargalit10} implies that there is a disk $B
\subset S$ with exactly two outward corners $x$ and $y$ so that $B
\cap (\alpha \cup \beta) = \bdy B$.  Suppose that $x$ is a {\em dual
intersection}: an intersection of a tie of $\alpha$ and a carried arc
of $\beta$.  See \reffig{DualIntersection}.

\begin{figure}[htbp]
\labellist
\small\hair 2pt
\pinlabel {$\alpha$} [bl] at 228.4 120.4
\pinlabel {$\beta$} [l] at 291.6 20.6
\endlabellist
$$\begin{array}{c}
\includegraphics[height=3.5cm]{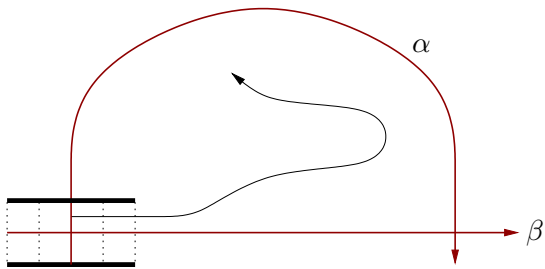} 
\end{array}$$
\caption{The left corner is a dual intersection, between a tie of
  $\alpha$ and a carried arc of $\beta$.  If the half-route $\rho$
  exits through $\alpha$ or $\beta$ then a bigon or non-trivial trigon
  is created.}
\label{Fig:DualIntersection}
\end{figure}

Let $\alpha' = \alpha \cap B$ and $\beta' = \beta \cap B$.  Orient
$\beta'$ away from $x$.  Let $z \in \alpha'$ be immediately adjacent
to $x$.  Without loss of generality we may assume that $z$ is to the
left of $\beta'$, near $x$.  Let $\rho$ be the half-route starting at
$z$, initially agreeing with $\beta$, and turning left at every
switch.  If $\rho \subset B$ then eventually $\rho$ repeats a branch
$b$ in the same direction; it follows that there is a curve $\gamma
\carr \tau$ contained in $B$ contradicting \reflem{Essential}.
However, if $\rho$ exits $B$ through $\alpha'$ ($\beta'$) then we
contradict efficient position of $\alpha$ ($\beta$).

It follows that the corner $x$ either lies in $S \setminus N$ or is
the intersection of carried arcs of $\alpha$ and $\beta$.  The same
holds for $y$.  In either case we cut off of $B$ a small neighborhood
of $x$ and of $y$: when the corner lies in $N$ we use a subarc of a
tie to do the cutting.  The result $B'$ is a rectangle with the
components of $\bdy_h B$ contained in $\alpha$ and $\beta$
respectively.  As $\ind(B') = 0$ the argument given in the case of an
annulus gives a sequence of rectangle swaps moving $\alpha$ across
$B$.
This reduces $|\alpha \cap \beta|$ by two and so completes the
induction step. 

The proof when $\alpha$ and $\beta$ are arcs follows the above but
omitting any mention of annulus swaps.

Finally, suppose that $\Delta, \Gamma$ are isotopic multicurves, both
in efficient position.  We may isotope $\Gamma$ to $\Delta$, as above,
being careful to always use innermost bigons.  This completes the
proof that efficient position is unique.

We end this subsection with a useful corollary:

\begin{corollary}
\label{Cor:EfficientIntersection}
Suppose that $\Gamma \subset \AC(S)$ is a finite collection of arcs
and curves in efficient position.  Then we may perform a sequence of
rectangle swaps to realize the pairwise geometric intersection
numbers.
\end{corollary}

\begin{proof}
Let $\Gamma = \{ \gamma_i \}_{i=1}^{k}$.  By induction, the curves of
$\Gamma' = \Gamma \setminus \{ \gamma_k \}$ realize their pairwise
geometric intersection numbers.  If $\gamma_k$ meets some $\gamma_i
\in \Gamma'$ non-minimally then by the Bigon Criterion~\cite[page
46]{FLP91} there is an innermost bigon between $\gamma_k$ and some
$\gamma_j \in \Gamma'$.  We now may reduce the intersection number
following the proof of uniqueness of efficient position.
\end{proof}



\subsection{Existence of efficient position}
\label{Sec:Exists}

Our hypotheses are weaker, and thus our discussion is more detailed,
but the heart of the matter is inspired by~\cite[pages
122-123]{MasurMinsky99}.

We may assume that $\tau$ fills $S$; for if not we replace $S$ by the
subsurface filled by $\tau$.  Since $\tau$ is transversely recurrent
for any $\epsilon, L > 0$ there is a finite area hyperbolic metric on
the interior of $S$ and an isotopy of $\tau$ so that: every branch of
$\tau$ has length at least $L$ and every train-route $\rho \carr \tau$
has geodesic curvature less than $\epsilon$ at every
point~\cite[Theorem~1.4.3]{PennerHarer92}.

Let $\tau^\HH$ be the lift of $\tau$ to $\HH = \HH^2$, the universal
cover of $S$.  Every train-route $\rho \carr \tau^\HH$ cuts $\HH$ into
a pair $H^\pm(\rho)$ of open {\em half-planes}.  Fix a route $\rho
\carr \tau^\HH$ and a half-branch $b' \subset \tau^\HH$ so that there
is some $n \in \ZZ$ with $b' \cap \rho = \rho(n)$.  We say the branch
$b$ is {\em rising} or {\em falling} with respect to $\rho$ as the
large half-branch at the switch $\rho(n)$ is contained in
$\rho|[n,\infty)$ or contained in $\rho|(-\infty,n]$.

\begin{claim}
\label{Clm:RisingFalling}
For any route $\rho \carr \tau^\HH$ one side of $\rho$ has infinitely
many rising branches while the other side has infinitely many falling
branches.
\end{claim}

\begin{proof}
Note that there are infinitely many half-branches on both sides of
$\rho$: if not then $\bdy_h N(\tau)$ would have a component without
corners, contrary to assumption.
Suppose that there are only finitely many rising branches along
$\rho$.  Then there is a curve $\gamma \carr \tau$ so that
$w_\gamma(b) \leq 1$ for every branch $b$ and so that the two sides of
$\gamma$ are combed in opposite directions.  
Thus $\tau$ is not recurrent, a contradiction.  The same contradiction
is obtained if there are only finitely many falling branches along
$\rho$.
\end{proof}

\begin{claim}
\label{Clm:Converge}
For any route $\rho \carr \tau^\HH$ and for any family of half-routes
$\{ \beta_n \}$ if $\beta_n \cap \rho = \rho(n)$ then $\lim_{n \to
\infty} \beta_n(\infty) = \rho(\infty)$.
\end{claim}

\begin{proof}
Let $x = \rho(\infty) \in \bdy_\infty \HH$.  Consider the subsequence
$\{ \beta_n \}$ where the first branch of each $\beta_n$ is falling.
Let $P_n = \rho|(-\infty, n] \cup \beta_n$, oriented away from
$\rho(-\infty)$.  Note that $P_n(\infty) = \beta_n(\infty)$.  Recall
that $\rho$ and $P_n$ are both uniformly close to
geodesics~\cite[pages 61--62]{PennerHarer92}.  Thus $P_n(\infty) \to
x$ as $n \to \infty$.

Now consider the subsequence $\{ \beta_n \}$ where the first branch of
each $\beta_n$ is rising.  Let $P_n = \beta_n \cup \rho|[n, \infty)$
oriented towards $x$; so $P_n(\infty) = x$ for all $n$.  Note that
$P(-\infty) = \beta_n(\infty)$.  Since all complementary regions of
$\tau^\HH$ have negative index none of the $P_n$ may cross each other.
It follows that either the $P_n$ exit compact subsets of $\HH$, and we
are done, or the $P_n$ converge~\cite[Theorem~1.5.4]{PennerHarer92} to
$P$, a train-route with $P(\infty) = x$.  Since $P$ does not cross any
$P_n$ deduce that $P$ and $\rho$ are disjoint.  But this
contradicts~\cite[Corollary 3.3.4]{Mosher03}: train-routes that share
an endpoint must share a half-route.
\end{proof}

Given distinct points $x, y, z \in S^1 = \bdy_\infty \HH$, arranged
counterclockwise, let $(y,z)$ be the component of $S^1 \setminus \{ y,
z \}$ that does not contain $x$. Let $[y,z]$ be the closure of $(y,
z)$.  Thus $x \in (z,y)$, $(y,z) \cap [z,y] = \emptyset$, and $(y,z)
\cup [z,y] = S^1$.

\begin{claim}
\label{Clm:Cofinal}
For any distinct $z, y \in S^1$ there is a train-route $\rho$ so that
one of the intervals $\bdy_\infty H^\pm(\rho)$ is contained in
$(z,y)$.
\end{claim}

\begin{proof}
The endpoints of train-routes are dense in $S^1 = \bdy_\infty \HH$.
Fix $x \in (z, y)$ so that $x$ is the endpoint of a train-route
$\gamma$.  Since there are infinitely many rising branches along
$\gamma$ (\refclm{RisingFalling}) the claim follows from the rising
case of \refclm{Converge}.
\end{proof}

Let $H_{x,y} \subset \HH$ be the convex hull of $(x, y) \subset
\bdy_\infty \HH$.  Let $\calH_{x,y}$ be the union of all open
half-planes $H(\rho)$ so that $\bdy_\infty H(\rho) \subset (x,y)$.
Since train-routes have geodesic curvature less than $\epsilon$ at
every point:

\begin{claim}
\label{Clm:NearlyHalfSpace}
The union $\calH_{x,y}$ is contained in an $\delta$--neighborhood of
$H_{x,y}$, where $\delta$ may be taken as small as desired by choosing
appropriate $\epsilon, L$.  \qed
\end{claim}

A set $X \subset \HH$ is {\em $\epsilon'$--convex} if every pair of
points in $X$ can be connected by a path in $X$ which has geodesic
curvature less than $\epsilon'$ at every point.

\begin{claim}
\label{Clm:Convex}
$\HH \setminus \calH_{x,y}$ is closed and $\epsilon'$--convex, where
$\epsilon'$ may be taken as small as desired by choosing appropriate
$\epsilon, L$.
\end{claim}

\begin{proof}
This is proved in detail on pages 122-123 of~\cite{MasurMinsky99}.
\end{proof}

\begin{claim}
\label{Clm:Accumulation}
The point $x$ is an accumulation point of $\bdy (\HH \setminus
\calH_{x,y})$.
\end{claim}

\begin{proof}
Pick a sequence of subintervals $(x_n, y_n) \subset (x,y)$ so that
$x_n, y_n \to x$ as $n \to \infty$.  By \refclm{Cofinal} for every $n$
there is a route $\rho_n$ and a half-plane $H_n = H(\rho_n)$ so that
$\bdy_\infty H_n \subset (x_n, y_n)$.  It follows that $H_n \subset
\calH_{x,y}$.  Let $r_n$ be any bi-infinite geodesic perpendicular to
$\bdy H_{x,y}$ and meeting $H_n$.  Thus $r_n \to x$ as $n \to \infty$.
By \refclm{NearlyHalfSpace} the intersection $r_n \cap \bdy (\HH
\setminus \calH_{x,y})$ is nonempty, and we are done.
\end{proof}

The next lemma is not needed for the proof of \refthm{EffPos}: we
state it and give the proof in order to introduce necessary techniques
and terminology. 

\begin{lemma}
\label{Lem:Basis}
For any non-parabolic point $x \subset S^1$ there is a sequence of
train-routes $\{ \rho_n \}$ with associated half-planes $\{ H(\rho_n)
\}$ forming a neighborhood basis for $x$.
\end{lemma}

\begin{proof}
Let $y, z$ be arbitrary points of $S^1$ so that $x, y, z$ are ordered
counterclockwise.  It suffices to construct a train-route separating
$x$ from $(y,z)$.

First assume that $x$ is the endpoint of a route $\rho$.
\refclm{RisingFalling} implies that there are infinitely many rising
branches $\{ a_m \}$ on one side of $\rho$ and infinitely many falling
branches $\{ c_n \}$ on the other side.  Run half-routes $\alpha_m$
and $\gamma_n$ through $a_m$ and $c_n$; so each half-route meets
$\rho$ in a single switch.  By \refclm{Converge} the endpoints
converge $\alpha_m(\infty), \gamma_n(\infty) \to x$.  Thus
sufficiently large $m, n$ give a train-route
\[
\alpha_m \cup \rho|_{[m, n]} \cup \gamma_n
\]
that separates $x$ from $(y,z)$, as desired. 

For the general case consider
\[
\calK = \HH \setminus (\calH_{z,x} \cup \calH_{x,y}).
\]
Note that $x$ is an accumulation point of $\calK$ (by
\refclm{Accumulation} and because $\calH_{z,x}$ cannot contain points
of $\bdy (\HH \setminus \calH_{x,y})$).
Fix any basepoint $w \in \calK$.  By \refclm{Convex} $\calK$ is
$\epsilon'$--convex.  Thus there is a path $r \subset \calK$ from $w$
to $x$ which has geodesic curvature less than $\epsilon'$ at every
point.  Since $x$ is not a parabolic point the projection of $r$ to
$S$ recurs to the thick part of $S$; thus $r$ meets infinitely many
branches $\{ b_n \}$ of $\tau^\HH$.

Suppose that $b$ is a branch of $\tau^\HH$ lying in $\calK$.  If
the two sides of $b$ meet $\calH_{z,x}$ and $\calH_{x,y}$ then $b$ is
a {\em bridge} of $\calK$.  If the sides of $b$ meet neither
$\calH_{z,x}$ nor $\calH_{x,y}$ then $b$ is an {\em interior branch}
of $\calK$.  If exactly one side of $b$ lies in $\calK$ then $b$ is a
{\em boundary branch}.  If both sides lie in $\calH_{z,x}$ (or both
sides lie in $\calH_{x,y}$) then $b$ is an {\em exterior branch}.
See \reffig{K} below. 

By convexity the path $r$ is disjoint from the exterior branches.
After a small isotopy the path $r$ is also disjoint from the boundary
branches, meets the interior branches transversely, and still has
geodesic curvature less than $\epsilon'$ at every point.

Now, if $r$ travels along a bridge then there are routes $\rho^\pm$
cutting off half-planes $H^\pm$ lying in $\calH_{z,x}$ and
$\calH_{x,y}$ respectively.  Then either $x$ is the endpoint of a
train-route or a cut and paste of $\rho^\pm$ gives the desired route
$\Gamma$ separating $x$ from $(y,z)$.  In either case we are done.

So suppose that $r$ only meets interior branches $\{ b_n \}$ of
$\calK$.  Let $\gamma_n$ be any train-route travelling along $b_n$.
If any of the $\gamma_n$ land at $x$ we are done, as above.  Supposing
not: Fixing orientations and passing to a subsequence we may assume
that $\gamma_n(\infty) \to x$ as $n \to \infty$.  There are now two
cases: suppose that for infinitely many $n$ we find that
$\gamma_n(-\infty) \in [y,z]$.  Then passing to a further subsequence
we have that $\gamma_n \to \Gamma$ where $\Gamma(\infty) =
x$~\cite[Theorem~1.5.4]{PennerHarer92}; thus $x$ is the endpoint of a
train-route and we are done as above.  The other possibility is that
for some sufficiently large $n$ both endpoints $\gamma_n(\pm\infty)$
lie in $(z, y)$.  Since $b_n$ is an interior branch $\gamma_n$
separates $x$ from $(y,z)$ and \reflem{Basis} is proved. 
\end{proof}

\subsection{Finding invariant efficient position}


Fix $\alpha \in \calC(S)$.  (The case where $\alpha \in \calA(S)$ is
dealt with at the end.)  Let $\alpha'$ be a component of the lift of
$\alpha$ to the universal cover $\HH$.  Let $\pi_1(\alpha)$ be the
cyclic subgroup (of the deck group) preserving $\alpha'$.  Let $\{ x,
y \} = \bdy_\infty \alpha' \subset S^1$.  We take
\[
\calK = \HH \setminus (\calH_{y,x} \cup \calH_{x,y}).
\]
By construction $\calK$ is $\pi_1(\alpha)$--invariant.  By
Claims~\ref{Clm:Convex} and~\ref{Clm:Accumulation} the set $\calK$ is
closed, $\epsilon'$--convex, and has $\{ x, y \} \subset \bdy_\infty
\calK$.  By \reflem{Basis} the only non-parabolic points of
$\bdy_\infty \calK$ are $x$ and $y$.  As in the proof of
\reflem{Basis} we find a bi-infinite path $r \subset \calK$ connecting
$y$ to $x$, with geodesic curvature less than $\epsilon'$ at every
point.  See \reffig{K}.

\begin{figure}[htbp]
\labellist
\small\hair 2pt
\pinlabel {$y$} [b] at 146 292.1
\pinlabel {$x$} [t] at 146 0
\pinlabel {$\calH_{x,y}$} [B] at 217.8 116.4
\pinlabel {$\calH_{y,x}$} [B] at 59.7 167.1
\endlabellist
$$\begin{array}{c}
\includegraphics[height=7cm]{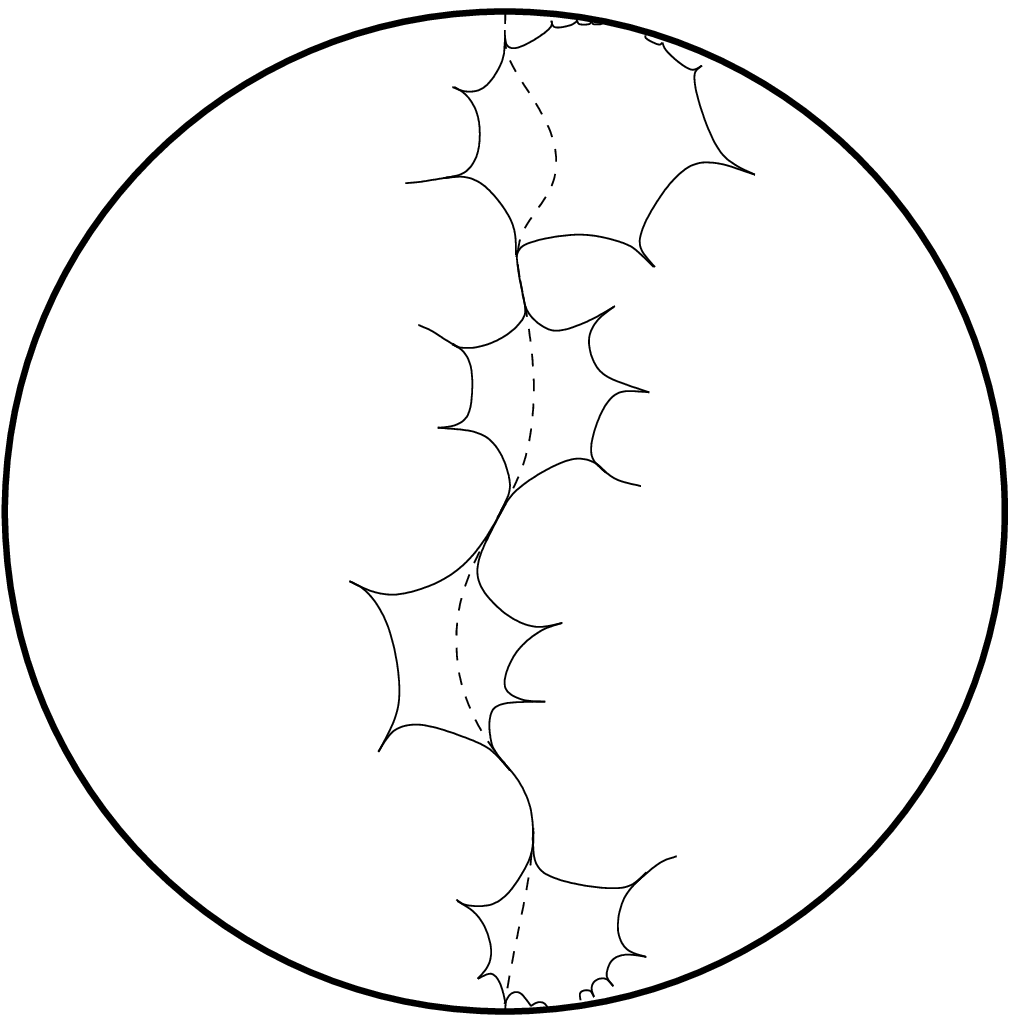} 
\end{array}$$
\caption{The path $r$ runs from $y$ to $x$.  To simplify the figure,
  no interior branches are shown.}
\label{Fig:K}
\end{figure}

Let $H(r)$ be the open half-plane to the right of $r$.  If we remove
the union
\[
\bigcup_{g \in \pi_1(\alpha)} g \cdot H(r)
\]
from $\HH$ then, as with \refclm{Convex}, what remains is closed and
$\epsilon''$--convex for some small $\epsilon''$.  It follows that we
may homotope the path $r$ to become a $\pi_1(\alpha)$--invariant
smooth path, contained in $\calK$ and transverse to the interior
branches, and avoiding the exterior branches of $\calK$.  A further
equivariant isotopy ensures that $r$ also avoids the boundary branches
of $\calK$.  Orient $r$ from $y$ to $x$.

\begin{remark}
\label{Rem:Disjoint}
Suppose that $\gamma \carr \tau^\HH$ is a train-route that separates
$x$ from $y$.  Note that if there exists a non-identity element $g \in
\pi_1(\alpha)$ so that $\gamma$ and $g \cdot \gamma$ meet then $r$ is
carried by $\tau^\HH$, thus $\alpha \carr \tau$, and we are done.  We
will henceforth assume that train-routes separating $x$ from $y$ are
disjoint from their non-trivial translates.
\end{remark}

Let $b$ be any interior branch of $\calK$ and let $\gamma$ be any
train-route travelling along $b$.  Since $b$ is interior, $\gamma$
must separate $x$ from $y$.  Orient $\gamma$ from $\calH_{y,x}$ to
$\calH_{x,y}$.  (Thus if $\gamma$ and $r$ meet once then the tangent
vectors to $r$ and $\gamma$, in that order, form a positive frame.)
The orientation of $\gamma$ gives an orientation to $b$.  Moreover, as
$b$ is an interior branch a cut and paste argument shows that the
orientation on $b$ is independent of our choice of $\gamma$.  Orient
all interior branches in this fashion and note that these orientations
agree at interior switches.

We say that $p \in r \cap b$ has {\em positive} or {\em negative sign}
as the tangent vectors to $r$ and $b$ (in that order) form a positive
or negative frame.  Suppose that there are $\Orbit \in \NN$ orbits of
points of negative sign, under the action of $\pi_1(\alpha)$.  We now
induct on $\Orbit$.

Suppose that $\Orbit$ is zero.  Any bigon between $r$ and a
train-route is contained in $\calK$ and so contributes one point of
positive and one point of negative sign.  So if there are no points of
negative sign then there are no bigons and $r$ is in efficient
position with respect to $\tau^\HH$.  Recall that $r$ is
$\pi_1(\alpha)$--invariant.  So $\beta \subset S$, the image of $r$
under the universal covering map, is an immersed curve in $S$
homotopic to $\alpha$.  If $\beta$ is embedded then we are done.  If
not then the Bigon Criterion for immersed curves~\cite{Thurston10}
implies that $\beta$ must have either a monogon or a bigon of
self-intersection.  If $\beta$ has a monogon $B$ of self-intersection
then, since index is additive, $\tau$ must be disjoint from $B$.  Thus
we can homotope $\beta$ to remove $B$ while fixing $\tau$ pointwise.
If $\beta$ has a bigon $B$ of self-intersection then, as in the proof
of uniqueness in \refsec{Unique}, we may remove $B$ via a sequence of
rectangle swaps.
After removing all monogons and bigons of self-intersection the curve
$\beta$ is embedded and in efficient position.

\begin{figure}[htbp]
\labellist
\small\hair 2pt
\pinlabel {$r$} [l] at 98.9 19.1
\pinlabel {$r$} [l] at 315.7 19.1
\pinlabel {$\gamma_L$} [l] at 148.1 183.2
\pinlabel {$\gamma_R$} [l] at 365.9 56.7
\endlabellist
$$\begin{array}{c}
\includegraphics[height=4cm]{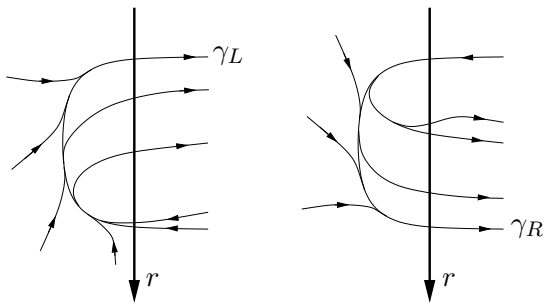} 
\end{array}$$
\caption{Left: The lowest point shown has negative sign.  The paths $r$
  and $\gamma_L$ form a bigon. Right: The corresponding figure for
  $\gamma_R$.} 
\label{Fig:Bigons}
\end{figure}

Suppose that $\Orbit$ is positive.  Let $b$ be a branch with a point
$p \in r \cap b$ of negative sign.  Let $\gamma_R$ ($\gamma_L$) be the
half-route starting at $p$, travelling in the direction of $b$, and
thereafter turning only right (left).  Each of $\gamma_R$ and
$\gamma_L$ must have at least one bigon with $r$, as their points at
infinity lie in $(x,y)$.  There are now two (essentially identical)
cases:
\begin{itemize}
\item
There is a bigon $B$ between $r$ and $\gamma_R$, to the right of $r$,
so that the corners of $B$ appear in the same order along $r$ and
$\gamma_R$.
\item
There is a bigon $B$ between $r$ and $\gamma_L$, to the right of $r$,
so that the corners of $B$ appear in opposite order along $r$ and
$\gamma_L$.
\end{itemize}
See \reffig{Bigons}.  If neither case holds then any half-route ending
at $p$ must originate in $(x,y)$, contradicting the fact that $p$ has
negative sign.
Note that, by \refrem{Disjoint}, $\gamma_R$ ($\gamma_L$) is disjoint
from its non-trivial $\pi_1(\alpha)$ translates.  It follows that the
bigon $B$ is also disjoint from its non-trivial translates.  Finally,
we may equivariantly isotope $r$ across $\pi_1(\alpha) \cdot B$.
Since the arc of $\gamma_R \cap \bdy B$ (respectively $\gamma_L \cap
\bdy B$) is combed outside of $B$ this isotopy reduces $\Orbit$ by at
least one. (Again, see \reffig{Bigons}.)  This completes the proof of
\refthm{EffPos} when $\alpha$ is a curve.

\subsection{Efficient position for arcs and multicurves}

Now suppose that $\alpha \subset S$ is an essential arc.  Let
$\alpha'$ be a lift of $\alpha$ to $\HH$, the universal cover of $S$.
Note that $\{ x, y \} = \bdy_\infty \alpha'$ is a pair of parabolic
points.  Construct $\calK = \HH \setminus (\calH_{y,x} \cup
\calH_{x,y})$ as before.  The proof now proceeds as above, omitting
any mention of $\pi_1(\alpha)$, equivariance, or annulus swaps. 

Finally suppose that $\Delta$ is a multicurve.  As shown in
\refsec{Exists} we may isotope, individually, every $\alpha \in
\Delta$ into efficient position.  By \refcor{EfficientIntersection}
all $\alpha \in \Delta$ may be realized disjointly in efficient
position.  This completes the proof of \refthm{EffPos}. \qed

\section{The structure theorem}
\label{Sec:Structure}

\subsection{Bounding diameter}

We now bound the diameters of the sets of wide arcs and curves carried
by the induced track.

\begin{lemma}
\label{Lem:AnnulusBound}
Suppose that $\tau$ is a birecurrent track and $X \subset S$ is an
essential annulus.  If $\tau|X \neq \emptyset$ then the diameter of
$V(\tau|X) \cup V^*(\tau|X)$ inside of $\calA(X)$ is at most eight.
\end{lemma}


\begin{proof}
In the proof we use $V, V^*$ to represent $V(\tau|X)$ and
$V^*(\tau|X)$.  Since $\tau|X \neq \emptyset$ it follows that
$\calA(\tau|X)$ is nonempty.  The first conclusion of
\reflem{WideInAnnuli} now implies that $V$ is nonempty.

\begin{claim*}
$V^* \neq \emptyset$.
\end{claim*}

\begin{proof}
By \reflem{ManyDuals} there is a dual curve $\beta \in \calC^*(\tau)$
so that $i(\alpha, \beta) > 0$.  Thus there is a lift $\beta' \subset
S^X$ with closure an essential arc.  Since $\tau|X \subset \tau^X$ it
follows that $\beta' \in \calA^*(\tau|X)$.  The first conclusion of
\reflem{WideInAnnuli} now implies that $V^*$ is nonempty.
\end{proof}

\begin{claim*}
If $\beta \in V$ and $\gamma \in V^*$ then $i(\beta, \gamma) \leq 3$.
\end{claim*}

\begin{proof}
Suppose that $i(\beta, \gamma) = n \geq 4$. Let $\{ \gamma_i
\}_{i=1}^{n-1}$ be the components of $\gamma \setminus \beta$ with
both endpoints on $\beta$. Let $R_i$ be the components of $S^X
\setminus (\beta \cup \gamma)$ with compact closure.  We arrange
matters so that opposite sides of $R_i$ are on $\gamma_i$ and
$\gamma_{i+1}$.  Let $R$ be the union of the $R_i$.  Since $\ind(R) =
0$ every region $T$ of the closure of $R \setminus N(\tau|X)$ also has
index zero and so is a rectangle.  If $T$ meets both $\gamma_i$ and
$\gamma_{i+1}$ then $\gamma$ was not wide, a contradiction.  As $n - 1
\geq 3$ any region $T$ meeting $\gamma_2$ is a compact rectangle
component of the closure of $S^X \setminus N(\tau|X)$.  An index
argument implies that $\tau^X$ and thus $\tau \subset S$ has a
complementary region with non-negative index, a contradiction.
\end{proof}

\noindent
Since $V, V^*$ are nonempty it follows that $\diam(V \cup V^*) \leq
8$.
\end{proof}

Now suppose that $X$ is not an annulus.  Prompted by
\reflem{VerticesWide} we define
\[
W(\tau^X) = \left\{ \alpha \in \AC\big(\tau^X\big) \st
           \mbox{$\alpha$ is wide} \right\}.
\]
Define $W^*(\tau^X)$ similarly, replacing $\AC(\tau^X)$ by
$\AC^*(\tau^X)$.

\begin{lemma}
\label{Lem:GeneralBound}
There is a constant $\Wide = \Wide(S)$ with the following property.
Suppose that $\tau$ is a track and $X \subset S$ is an essential
subsurface (not an annulus) with $\pi_X(\tau) \neq \emptyset$.  Then
the diameter of $W(\tau^X) \cup W^*(\tau^X)$ inside of $\AC(X)$ is at
most $\Wide$.  Furthermore if, after isotoping $\bdy X$ into efficient
position, the induced orientation on $\bdy X$ is not wide then
either $\calC(\tau|X)$ or $\calC^*(\tau|X)$ has diameter at most two
in $\calC(X)$.
\end{lemma}


\begin{proof}
In the proof we use $W, W^*, \AC, \AC^*$ to represent $W(\tau^X)$ and
so on.  Since $\pi_X(\tau) \neq \emptyset$ there is some vertex cycle
$\alpha \in V(\tau)$ so that $\alpha$ cuts $X$.  Since $\alpha$ is
wide (\reflem{VerticesWide}) there is a lift $\alpha' \carr \tau^X$
which is also wide; deduce that $W$ is nonempty.


\begin{claim*}
$W^* \neq \emptyset$.
\end{claim*}

\begin{proof}
By \reflem{ManyDuals} there is a dual curve $\alpha \in \calC^*(\tau)$
cutting $X$.  By \reflem{Decompose} there is a wide dual $\beta$ that
also cuts $X$.  Thus there is a lift $\beta' \subset S^X$ with closure
an essential wide arc or wide essential, non-peripheral curve.  So
$\beta' \in W^*$ as desired.
\end{proof}

Now isotope $\bdy X$ into efficient position.  Let $X'$ be the compact
component of the preimage of $X$ under the covering map $S^X \to S$.
Note that $\bdy X'$ is in efficient position with respect to $\tau^X$.
Note that the covering map $S^X \to S$ induces a homeomorphism between
$X'$ and $X$.  Let $N^X = N(\tau^X) \subset S^X$ be the preimage of $N
= N(\tau)$.  Let $N' = X' \cap N^X$.  Again, the covering map induces
a homeomorphism between $N'$ and $N \cap X$.

Suppose that $\bdy X$, with its induced orientation, is not wide.  If
$\bdy X$ fails to be wide in $S \setminus N(\tau)$ then there is a
properly embedded, essential arc $\gamma \subset X$ disjoint from
$N(\tau)$.  Lift $\gamma$ to $\gamma' \subset X'$.  Adjoin to
$\gamma'$ geodesic rays in $S^X \setminus X'$ to obtain an essential,
properly embedded arc $\gamma'' \subset S^X$.  Note that $i(\gamma'',
\alpha) = 0$ for every $\alpha \in \AC$; only intersections in $X'$
contribute to geometric intersection number as computed in $S^X$.
This implies that $\diam_X(\calC(\tau|X)) \leq 2$.  Furthermore,
$i(\gamma'', \beta) \leq 2$ for every $\beta \in W^*$.  This gives the
desired diameter bound for $W \cup W^*$.

If, instead, $\bdy X$ fails to be wide in $N(\tau)$ then there is a
properly embedded, essential arc $\gamma \subset X$ that is a subarc
of a tie.  Again, lift and extend to an essential arc $\gamma''
\subset S^X$ so that $i(\gamma'', \beta) = 0$ for any $\beta \in
\AC^*$.  This implies that $\diam_X(\calC^*(\tau|X)) \leq 2$.  We also
have $i(\gamma'', \alpha) \leq 2$ for any $\alpha \in W$. Again the
diameter is bounded.

Now suppose that $\bdy X$ is wide.  Thus, for every $b \in
\calB(\tau)$, the rectangle $R_b$ meets $\bdy X$ in at most a pair of
arcs. It follows that $N \cap X$, and thus $N'$, is a union of at most
$2|\calB(\tau)|$ subrectangles of the form $R', R'' \subset R_b$.
Suppose that $\alpha \in W$ and $\beta \in W^*$.  Then $\alpha$ and
$\beta$ each meet a subrectangle $R'$ in at most two arcs.  Thus
$\alpha$ and $\beta$ intersect in at most four points inside of $R'$.
Thus $i(\alpha, \beta) \leq 8 |\calB(\tau)|$ and \reflem{GeneralBound}
is proved.
\end{proof}

\subsection{Accessible intervals}
\label{Sec:Access}

Suppose that $\{ \tau_i \}_{i=0}^N$ is a sliding and splitting
sequence of birecurrent train tracks.  Suppose $X \subset S$ is an
essential subsurface, yet not an annulus, with $\xi(X) < \xi(S)$.
Define
\begin{gather*}
m_X = 
   \min \big\{ i \in [0,N] \st \diam_X(\calC^*(\tau_i|X)) \geq 3 \big\} \\
\tag*{\text{and}}
n_X = 
   \max \big\{ i \in [0,N] \st \diam_X(\calC(\tau_i|X)) \geq 3 \big\}.
\end{gather*}
If either $m_X$ or $n_X$ is undefined or if $n_X < m_X$ then $I_X$,
the {\em accessible interval} is empty. Otherwise, $I_X = [m_X, n_X]$.

If $X \subset S$ is an annulus, then $I_X$ is defined by replacing
$\calC$ by $\calA$ and increasing the lower bound on diameter from 3
to 9.  We may now state the structure theorem:


\begin{theorem}
\label{Thm:FellowTravel}
For any surface $S$ with $\xi(S) \geq 1$ there is a constant $\Drift =
\Drift(S)$ with the following property: Suppose that $\{ \tau_i
\}_{i=0}^N$ is a sliding and splitting sequence of birecurrent train
tracks in $S$ and suppose that $X \subset S$ is an essential
subsurface.
\begin{itemize}
\item
For every $[a,b] \subset [0,N]$ if $[a,b] \cap I_X = \emptyset$ and
$\pi_X(\tau_b) \neq \emptyset$ then $d_X(\tau_a, \tau_b) \leq \Drift$.
\end{itemize}
Suppose $i \in I_X$.  If $X$ is an annulus:
\begin{itemize}
\item
The core curve $\alpha$ is carried by and wide in $\tau_i$.
\item
Both sides of $\alpha$ are combed in the induced track $\tau_i|X$.
\item
If $i+1 \in I_X$ then $\tau_{i+1}|X$ is obtained by taking subtracks,
slides, or at most a pair of splittings of $\tau_i|X$.
\end{itemize}
If $X$ is not an annulus:
\begin{itemize}
\item
When in efficient position
$\bdy X$ is wide with respect to $\tau_i$.
\item
The track $\tau_i|X$ is birecurrent and fills $X$.
\item
If $i+1 \in I_X$ then $\tau_{i+1}|X$ is either a subtrack, a slide, or
a split of $\tau_i|X$.
\end{itemize}
\end{theorem}


\begin{proof}
Fix an interval $[a,b] \subset [0,N]$.  Note that $\tau_b \carr
\tau_a$ and so $\tau_b^X \carr \tau_a^X$.  Thus $\AC(\tau_b^X) \subset
\AC(\tau_a^X)$ while $\AC^*(\tau_a^X) \subset \AC^*(\tau_b^X)$.

\begin{claim*}
If $[a,b] \cap I_X = \emptyset$ and $\pi_X(\tau_b) \neq \emptyset$
then $d_X(\tau_a, \tau_b) \leq \Drift$.
\end{claim*}

\begin{proof}
Fix, for the duration of the claim, a vertex cycle $\beta \in
V(\tau_b)$ so that $\beta$ cuts $X$.  Since $\beta$ is also carried by
$\tau_a$ there is, by \reflem{Decompose}, a vertex cycle $\alpha \in
V(\tau_a)$ cutting $X$.  Pick $\alpha' \in \pi_X(\alpha)$ and $\beta'
\in \pi_X(\beta)$.  Note \reflem{VerticesWide} implies that
$\alpha'$ is wide in $\tau_a^X$ while $\beta'$ is wide in $\tau_b^X$.
The proof divides into cases depending on the relative positions of
$a, b, m_X$ and $n_X$.

\begin{case}
Suppose $n_X < a$ or $n_X$ is undefined. 
\end{case}

Note that $\beta' \carr \tau_a^X$.  If $X$ is an annulus then since $a
\nin I_X$ the diameter of $\calA(\tau_a|X)$ is at most eight; thus
$d_X(\alpha, \beta) \leq 8$ and we are done.

Suppose that $X$ is not an annulus.  If $\beta'$ is an arc then
\reflem{Surgery} gives two cases: we may replace $\beta'$ by $\gamma$
which is either a wide arc in $\tau_a^X$ or is an essential
non-peripheral curve in $\tau_a|X$.  (If $\beta'$ is a curve then let
$\gamma = \beta'$.)  In either case \reflem{Surgery} ensures that
$i(\gamma, \beta') \leq 2$ and so $d_X(\gamma, \beta') \leq 4$.  If
$\gamma$ is an arc then both $\alpha'$ and $\gamma$ are wide so
\reflem{GeneralBound} gives $d_X(\alpha, \beta) \leq \Wide + 4$.  If
$\gamma$ is a curve pick any $\delta \in V(\tau_a|X)$.  Then
\reflem{GeneralBound} implies that $d_X(\alpha', \delta) \leq \Wide$.
Also, $a \nin I_X$ implies that $d_X(\delta, \gamma) \leq 2$.  Thus
$d_X(\alpha, \beta) \leq \Wide + 6$.

\begin{case}
Suppose $b < m_X$ or $m_X$ is undefined. 
\end{case}

If $X$ is an annulus, then \reflem{AnnulusBound} gives wide duals
$\alpha^* \in V^*(\tau_a|X)$ and $\beta^* \in V^*(\tau_b|X)$ so that
$d_X(\alpha', \alpha^*), d_X(\beta', \beta^*) \leq 8$.  It follows
that the arc $\alpha^*$ also lies in $\calA^*(\tau_b|X)$.  Since $b
\nin I_X$ we have $d_X(\alpha^*, \beta^*) \leq 8$.  Thus $d_X(\alpha,
\beta) \leq 24$, as desired.

If $X$ is not an annulus, then by \reflem{GeneralBound} there is a
wide dual $\alpha^* \in W^*(\tau_a^X)$ so that $d_X(\alpha', \alpha^*)
\leq \Wide$.  Again, $\alpha^*$ is also an element of
$\AC^*(\tau_b^X)$ but may not be wide there.  If $\alpha^*$ is an arc
then \reflem{Surgery} gives two cases: we may replace $\alpha^*$ by
$\gamma^*$ which is either a wide dual arc to $\tau_b^X$ or is an
essential non-peripheral dual curve to $\tau_b^X$.  (If $\alpha^*$ is
a curve then let $\gamma^* = \alpha^*$.)  So $i(\alpha^*, \gamma^*)
\leq 2$ and thus $d_X(\alpha^*, \gamma^*) \leq 4$.  If $\gamma^*$ is a
wide dual arc then \reflem{GeneralBound} implies that $d_X(\gamma^*,
\beta') \leq \Wide$ and so $d_X(\alpha, \beta) \leq 2\Wide + 4$.  If
$\gamma^*$ is a dual curve then, as $b \nin I_X$, any dual wide curve
$\delta^* \in V^*(\tau_b|X)$ has $d_X(\gamma^*, \delta^*) \leq 2$.
Again, \reflem{GeneralBound} implies that $d_X(\delta^*, \beta') \leq
\Wide$ and so $d_X(\alpha, \beta) \leq 2\Wide + 6$.

\begin{case}
Suppose $a \leq n_X < c < m_X \leq b$.
\end{case}

The first two cases bound $d_X(\tau_c, \tau_b)$ and $d_X(\tau_a,
\tau_c)$; thus we are done by the triangle inequality.

\begin{case}
Suppose $a \leq n_X < m_X \leq b$ and $m_X = n_X + 1$.
\end{case}

Let $c = n_X$ and $d = m_X$.  The first two cases bound $d_X(\tau_d,
\tau_b)$ and $d_X(\tau_a, \tau_c)$.  Since $V(\tau_c)$ and $V(\tau_d)$
have bounded intersection $d_X(\tau_c, \tau_d)$ is also bounded and
the claim is proved. 
\end{proof}

Now fix $i \in I_X$. 

\begin{claim*}
If $X$ is an annulus:
\begin{itemize}
\item
The core curve $\alpha$ is carried by and is wide in $\tau_i$.
\item
Both sides of $\alpha$ are combed in the induced track $\tau_i|X$.
\item
If $i+1 \in I_X$ then $\tau_{i+1}|X$ is obtained by taking subtracks,
slides, or at most a pair of splittings of $\tau_i|X$.
\end{itemize}
\end{claim*}

\begin{proof}
Since $i \in I_X$, both $\calA(\tau_i|X)$ and $\calA^*(\tau_i|X)$ have
diameter at least nine.  From \reflem{AnnulusBound} deduce that the
inclusions $V \subset \calA$ and $V^* \subset \calA^*$ are strict.
Thus by \reflem{AnnulusEquality} the core curve $\alpha$ is both
carried by and dual to $\tau_i|X$.  The second statement now follows
from \reflem{Combed}.  Thus at least one side of $\alpha$ is combed in
$\tau_i^X$.  Projecting from $S^X$ to $S$ we find that $\alpha \carr
\tau_i$.  If $\alpha$ is not wide in $\tau_i$ then we deduce that
neither side of $\alpha$ is combed in $\tau^X$, a contradiction.

Suppose that $\tau_i$ slides to $\tau_{i+1}$.  Then, up to isotopy,
$\tau_{i+1}$ slides to $\tau_i$.  Since slides do not kill essential
arcs it follows that $\tau_{i+1}|X$ is obtained from $\tau_i|X$ by an
at most countable collection of slides. 

Now suppose that $\tau_{i+1}$ is obtained by splitting $\tau_i$ along
a large branch $b$.  Thus $\tau_{i+1}^X$ is obtained from $\tau_i^X$
by splitting all of the countably many lifts of $b$.  Every essential
arc carried by $\tau_{i+1}|X$ is also carried by $\tau_i^X$.  Let
$\tau' \subset \tau_i^X$ be the union of these essential routes.  It
follows that $\tau_{i+1}|X$ is obtained from $\tau'$ by splitting
along lifts of $b$ that are also large branches of $\tau'$.  Since
both sides of $\alpha$ are combed in $\tau_i|X$ the same is true in
$\tau'$ and so any component of $\tau' \setminus \alpha$ is a tree
without large branches.  The track $\tau'$ therefore has only finitely
many large branches, all contained in $\alpha$.  Since $\alpha$ is
wide in $\tau_i$ there are at most two preimages of the large branch
$b$ contained in $\alpha \subset S^X$.  Thus $\tau_{i+1}|X$ is
obtained from $\tau'$ by at most two splittings.  This proves the
claim.  (See \reffig{Meet} for pictures of how $\alpha$ may be carried
by $\tau_i$ and how splitting $b$ effects $\tau_i|X$.)
\end{proof}

\begin{figure}[htbp]
$$\begin{array}{cc}
\includegraphics[height=3cm]{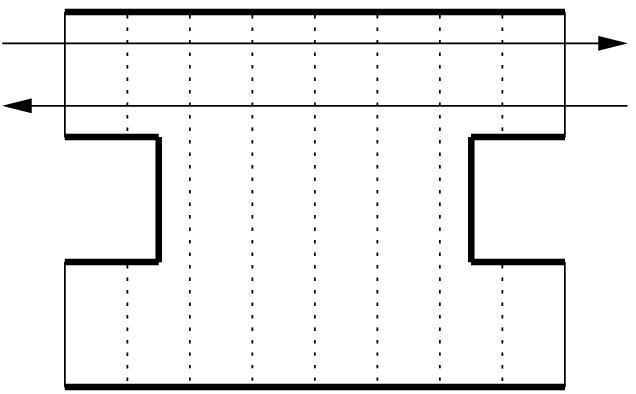} &
\includegraphics[height=3cm]{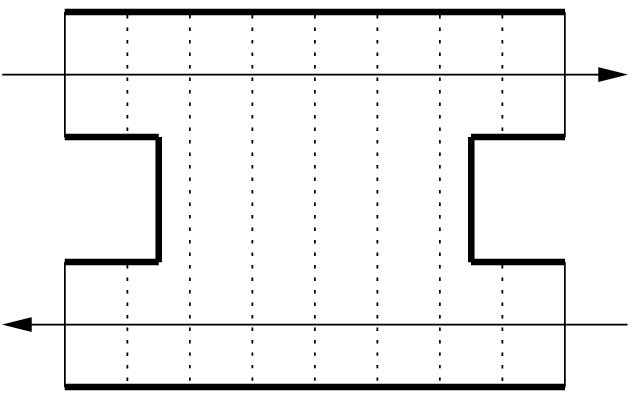} \\
\\
\includegraphics[height=3cm]{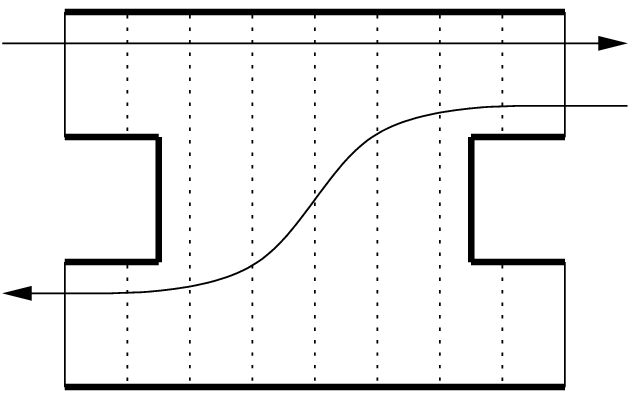} &
\includegraphics[height=3cm]{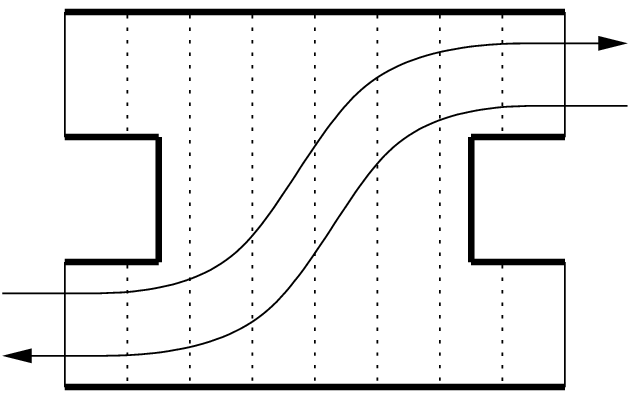} 
\end{array}$$
\caption{Four of the possible ways for an oriented, carried, wide
  curve $\alpha$ to meet a large rectangle $R_b$ of $N(\tau)$.  Note
  that when $X$ is an annulus and $\alpha$ is the core curve the upper
  left picture implies that neither side of $\alpha$ is combed in
  $\tau|X$.  Splitting the upper right deletes zero or two components
  of $\tau_i|X \setminus \alpha$.  In the bottom row, only the left
  splitting is possible when $i, i+1 \in I_X$.  On the bottom left one
  component of $\tau_i|X \setminus \alpha$ is deleted and $\tau_i|X$
  is split once.  On the bottom right $\tau_i|X$ is split twice.}
\label{Fig:Meet}
\end{figure}

\begin{claim*}
Suppose $X$ is not an annulus. 
\begin{itemize}
\item
When in efficient position
$\bdy X$ is wide with respect to $\tau_i$.
\item
The track $\tau_i|X$ is birecurrent and fills $X$.
\item
If $i+1 \in I_X$ then $\tau_{i+1}|X$ is either a subtrack, a slide, or
a split of $\tau_i|X$.
\end{itemize}
\end{claim*}

\begin{proof}
Since $i \in I_X$ the second conclusion of \reflem{GeneralBound}
implies that $\bdy X$ is wide.  The induced track $\tau_i|X$ carries a
pair of curves at distance at least three, so fills $X$.  Also,
$\tau_i|X$ is recurrent by definition.  For any branch $b' \in
\tau_i|X \subset S^X$, let $b \subset \tau$ be the image in $S$.
Since $\tau$ is transversely recurrent there is a dual curve $\beta$
meeting $b$.  Lifting $\beta$ to a curve or arc $\beta' \subset S^X$
gives a dual to $\tau|X$ meeting $b'$.  Thus $\tau|X$ is transversely
recurrent with respect to arcs and curves, as defined in
\refsec{InducedForNonAnnuli}

Now, if $\tau_i$ slides to $\tau_{i+1}$ then, as in the annulus case,
$\tau_i|X$ slides to $\tau_{i+1}|X$.  Suppose instead that $\tau_i$
splits to $\tau_{i+1}$ along the branch $b \in \calB(\tau_i)$.  Thus
$\tau_i|X$ splits (or isotopes) to a track $\tau'$ so that
$\tau_{i+1}|X$ is a subtrack.  Let $R_b \subset N(\tau_i)$ be the
rectangle corresponding to the branch $b$.  Isotope $\bdy X$ into
efficient position with respect to $N(\tau_i)$ and recall that $X$ is
to the left of $\bdy X$.  Note that, by an isotopy, we may arrange for
all curves in $\calC(\tau_i|X)$ to be disjoint from $\bdy X$.  Let
$\beta \subset R_b$ be a {\em central splitting arc}: a carried arc
completely contained in $R_b$.

If $\beta \cap X$ is empty then $\tau_i|X$ is identical to
$\tau_{i+1}|X$.  If $\beta \subset X$ then $\tau_{i+1}|X$ is either a
subtrack or a splitting of $\tau_i|X$, depending how the carried
curves of $\tau_i|X$ meet the lift of $R_b$.  In all other cases
$\tau_{i+1}|X$ is a subtrack of $\tau_i|X$.  See \reffig{Meet} for
some of the ways carried subarcs of $\bdy X$ may meet $R_b$.  If $\bdy
X \cap R_b$ contains a tie then $\tau_i|X$ is identical to
$\tau_{i+1}|X$.  This completes the proof of the claim.
\end{proof}

\noindent
Thus \refthm{FellowTravel} is proved. 
\end{proof}

We now rephrase a result of Masur and Minsky using the {\em refinement}
procedure of Penner and Harer~\cite[page 122]{PennerHarer92}.

\begin{theorem}
\cite[Theorem~1.3]{MasurMinsky04}
\label{Thm:Unparam}
For any surface $S$ with $\xi(S) \geq 1$ there is a constant $Q =
Q(S)$ with the following property: For any sliding and splitting
sequence $\{ \tau_i \}_{i=0}^N$ of birecurrent train tracks in $S$ the
sequence $\{ V(\tau_i) \}_{i=0}^N$ forms a $Q$--unparameterized
quasi-geodesic in $\calC(S)$.
\end{theorem}


\refthm{FellowTravel} implies that the same result holds after
subsurface projection.

\begin{theorem}
\label{Thm:LocalUnparam}
For any surface $S$ with $\xi(S) \geq 1$ there is a constant $Q =
Q(S)$ with the following property: For any sliding and splitting
sequence $\{ \tau_i \}_{i=0}^N$ of birecurrent train tracks in $S$ and
for any essential subsurface $X \subset S$ if $\pi_X(\tau_N) \neq
\emptyset$ then the sequence $\{ \pi_X(\tau_i) \}_{i=0}^N$ is a
$Q$--unparameterized quasi-geodesic in $\calC(X)$.
\end{theorem}

\begin{proof}
By the first conclusion of \refthm{FellowTravel} we may restrict
attention to the subinterval $[p,q] = I_X \subset [0,N]$.

Fix any vertex $\alpha \in V(\tau_q|X)$.  Note $\alpha$ is carried by
$\tau_i|X$ for all $i \leq q$.  So define $\sigma_i \subset \tau_i|X$
to be the minimal pretrack carrying $\alpha$.  Since $\sigma_i$ does
not carry any peripheral curves $\sigma_i$ is a train track.  Note
that $\sigma_i$ is recurrent by definition and transversely recurrent
by \reflem{TransverseRecurrenceAC}.  Applying \refthm{FellowTravel},
for all $i \in [p, q-1]$ the track $\sigma_{i+1}$ is a slide, a split,
or identical to the track $\sigma_i$.

\refthm{Unparam} implies the sequence $\{ V(\sigma_i) \}$ is a
$Q$--unparameterized quasi-geodesic in $\calC(X)$.  Note that
$d_X(\sigma_i, \tau_i|X)$ is uniformly bounded because $\sigma_i$ is a
subtrack.

Since $\alpha \carr \tau_i|X$ the curve $\alpha$ is also carried by
$\tau_i$.  By \reflem{Decompose} there is a vertex cycle $\beta_i
\carr \tau_i$ that cuts $X$.  Since $\beta_i$ is wide
(\reflem{VerticesWide}) any element $\beta_i' \in \pi_X(\beta)$ is
carried by and wide in $\tau_i^X$.  It follows that $\beta_i'$ and the
vertex cycles of $\tau_i|X$ have bounded intersection.  Thus
$d_X(\tau_i, \tau_i|X)$ is uniformly bounded and we are done.
\end{proof}


\section{Further applications of the structure theorem}
\label{Sec:Applications}

We now turn to Theorems~\ref{Thm:Quasi} and~\ref{Thm:TT}; both are
slight generalizations of a result of Hamenst\"adt~\cite[Corollary
3]{Hamenstadt09}. Our proofs, however, rely on \refthm{FellowTravel}
and are quite different from the proof found in~\cite{Hamenstadt09}.

\subsection{The marking and train track graphs}
Suppose that $S$ is not an annulus. A finite subset $\mu \subset
\AC(S)$ {\em fills} $S$ if for all $\beta \in \calC(S)$ there is a
$\gamma \in \mu$ so that $i(\beta, \gamma) \neq 0$.  If $\mu, \nu
\subset \AC(S)$ then we define
\[
i(\mu, \nu) \,\, 
   = \sum_{\alpha \in \mu, \beta \in \nu} i(\alpha, \beta).
\]
Also, let $i(\mu) = i(\mu, \mu)$ be the self-intersection number.  A
set $\mu$ is a {\em $k$--marking} if $\mu$ fills $S$ and $i(\mu) \leq
k$.  Two sets $\mu, \nu$ are {\em $\ell$--close} if $i(\mu, \nu) \leq
\ell$.

Define $k_0 = \max_\tau i(V(\tau))$, where $\tau$ ranges over tracks
with vertex cycles $V(\tau)$ filling $S$.  Define $\ell_0 =
\max_{\tau,\sigma} i(V(\tau), V(\sigma))$, where $\sigma$ ranges over
tracks obtained from $\tau$ by a single splitting.  Referring to
\cite{MasurMinsky00} for the necessary definitions, we define $k_1 =
\max_\mu i(\mu)$, where $\mu$ ranges over {\em complete clean
markings} of $S$.  Define $\ell_1 = \max_{\mu,\nu} i(\mu, \nu)$, where
$\nu$ ranges over markings obtained from $\mu$ by a single {\em
elementary move}.  Define $\ell_2 = \max_\tau \min_\mu i(V(\tau),
\mu)$.

Note that there are only finitely many tracks $\tau$ and finitely many
complete clean markings $\mu$, up to the action of $\MCG(S)$.  An
index argument bounds $|\calB(\tau)|$ and so bounds the number of
splittings of $\tau$.  Lemma~2.4 of~\cite{MasurMinsky00} bounds the
number of elementary moves for $\mu$. Thus the quantities $k_0, k_1,
\ell_0, \ell_1$ are well-defined.  An upper bound for $\ell_2$ can be
obtained by surgering $V(\tau)$ to obtain a complete clean marking:
see the discussion preceding Lemma~6.1 in~\cite{Behrstock06}.  Now
define $k = \max \{ k_0, k_1 \}$ and $\ell = \max \{ \ell_0, \ell_1,
\ell_2 \}$.  Define $\calM(S)$ to be the {\em marking graph}: the
vertices are $k$--markings and the edges are given by
$\ell$--closeness.  (When $S$ is an annulus we take $\calM(S) =
\calA(S)$.  Recall that $\calA(S)$ is quasi-isometric to $\MCG(S,
\bdy) \isom \ZZ$.)

That $\calM(S)$ is connected now follows from the discussion at the
beginning of~\cite[Section~6.4]{MasurMinsky00}.  Accordingly, define
$d_{\calM(S)}(\mu, \nu)$ to be the length of the shortest edge-path
between the markings $\mu$ and $\nu$.  

Since the above definitions are stated in terms of geometric
intersection number, the mapping class group $\MCG(S)$ acts via
isometry on $\calM(S)$.  Counting the appropriate set of ribbon graphs
proves that the action has finitely many orbits of vertices and edges.
The Alexander method~\cite[Section 2.4]{FarbMargalit10} proves that
vertex stabilizers are finite and hence the action is proper.  It now
follows from the Milnor-\v{S}varc Lemma~\cite[Proposition
I.8.19]{Bridson99} that any Cayley graph for $\MCG(S)$ is
quasi-isometric to $\calM(S)$.

Define $\calT(S)$, the {\em train track graph} as follows: vertices
are isotopy classes of birecurrent train tracks $\tau \subset S$ so
that $V(\tau)$ fills $S$.  Connect vertices $\tau$ and $\sigma$ by an
edge exactly when $\sigma$ is a slide or split of $\tau$.  Let
$d_{\calT(S)}(\tau, \upsilon)$ be the minimal number of edges in a
path in $\calT(S)$ connecting $\tau$ to $\upsilon$, if such a path
exists.  Note that the map $\tau \mapsto V(\tau)$ from $\calT(S)$ to
$\calM(S)$ sends edges to edges (or to vertices) and thus is distance
non-increasing.  For further discussion of graphs tightly related to
$\calT(S)$ see~\cite{Hamenstadt09}.

We adopt the following notations.  If $\{ \tau_i \}_{i=0}^N$ is a
sliding and splitting sequence in $\calT(S)$ and $I = [p,q] \subset
[0,N]$ is a subinterval then $|I| = q - p$ and $d_{\calT(S)}(I) =
d_{\calT(S)}(\tau_p, \tau_q)$.  If $\tau, \sigma \in \calT(S)$ then
define
\[
d_{\calM(X)}(\tau, \sigma) = 
        d_{\calM(X)}\big(V(\tau|X), V(\sigma|X) \big). 
\]
Also take $d_{\calM(X)}(I) = d_{\calM(X)}(\tau_p, \tau_q)$.  

\begin{theorem}
\label{Thm:Quasi}
For any surface $S$ with $\xi(S) \geq 1$ there is a constant $Q =
Q(S)$ with the following property: Suppose that $\{ \tau_i \}_{i=0}^N$
is a sliding and splitting sequence in $\calT(S)$.  Then the sequence
$\{ V(\tau_i) \}_{i=0}^N$, as parameterized by splittings, is a
$Q$--quasi-geodesic in the marking graph.
\end{theorem}


Our final generalization of~\cite[Corollary 3]{Hamenstadt09} follows
from \refthm{Quasi}:

\begin{theorem}
\label{Thm:TT}
For any surface $S$ with $\xi(S) \geq 1$ there is a constant $Q =
Q(S)$ with the following property: If $\{ \tau_i \}_{i=0}^N$ is a
sliding and splitting sequence in $\calT(S)$, injective on slide
subsequences, then $\{ \tau_i \}$ is a $Q$--quasi-geodesic.
\end{theorem}

Notice that here, unlike \refthm{Quasi}, the parameterization is by
index. We require:

\begin{lemma}
\label{Lem:Slide}
There is a constant $A = A(S)$ so that if $\{ \tau_i \}_{i = 1}^N$ is
an injective sliding sequence in $\calT(S)$ then $N + 1 \leq A$. \qed
\end{lemma}

\begin{proof}[Proof of \refthm{TT}]
Let $\{ \tau_i \}$ be the given sliding and splitting sequence in
$\calT(S)$.  Let $I = [p,q] \subset [0,N]$ be a subinterval.  Note
that $d_{\calT(S)}(I) \leq |I|$ because $\{ \tau_i \}$ is an edge-path
in $\calT(S)$.

Define $\calS(I)$ to be the set of indices $r \in I$ where $r + 1 \in
I$ and $\tau_{r+1}$ is a splitting of $\tau_r$.  Thus $|I| \leq_A
|\calS(I)|$, where $A$ is the constant of \reflem{Slide}.  Now,
\refthm{Quasi} implies $|\calS(I)| \leq_Q d_{\calM(S)}(I)$.  Finally,
since $\tau \mapsto V(\tau)$ is distance non-increasing we have
$d_{\calM(S)}(I) \leq d_{\calT(S)}(I)$.  Deduce that $|I| \leq_Q
d_{\calM(S)}(I)$, for a somewhat larger value of $Q$.
\end{proof}

\begin{remark}
\label{Rem:Connected}
Note that we have not used the connectedness of $\calT(S)$, an issue
that appears to be difficult to approach combinatorially.  For a proof
of connectedness see \cite[Corollary 3.7]{Hamenstadt09}.
\end{remark}

\subsection{Hyperbolicity and the distance estimate}

To prove \refthm{Quasi} we will need:

\begin{theorem}
\cite[Theorem~1.1]{MasurMinsky99}
\label{Thm:Hyperbolic}
For every connected compact orientable surface $X$ there is a constant
$\delta_X$ so that $\calC(X)$ is $\delta_X$--hyperbolic.  
\end{theorem}

An important consequence of the Morse Lemma~\cite[Theorem
III.H.1.7]{Bridson99} is a reverse triangle inequality. 

\begin{lemma}
\label{Lem:Reverse}
For every $\delta$ and $Q$, there is a constant $\Reverse =
\Reverse(\delta, Q)$ with the following property: For any
$\delta$--hyperbolic space $\calX$, for any $Q$--unparameterized
quasi-geodesic $f \from [m,n] \to \calX$, and for any $a, b, c \in
[m,n]$, if $a \leq b \leq c$ then
\[
d_\calX(\alpha, \beta) + d_\calX(\beta, \gamma) 
    \leq d_\calX(\alpha, \gamma) + \Reverse
\]
where $\alpha, \beta, \gamma = f(a), f(b), f(c)$. \qed
\end{lemma}

We now take $\Reverse = \Reverse(\delta, Q(S))$ where $\delta = \max
\{ \delta_X \st X \subset S \}$, as provided by \refthm{Hyperbolic},
and $Q(S)$ is the constant of \refthm{LocalUnparam}.  The next central
result needed is the {\em distance estimate} for $\calM(S)$.  Let
$[\cdot]_C$ be the {\em cut-off function}:
\[
[x]_C \, = \, 
\left\{
\begin{array}{cc}
0, & \mbox{if $x < C$} \\
x, & \mbox{if $x \geq C$} 
\end{array}
\right\}.
\]
We may now state the distance estimate:

\begin{theorem}
\cite[Theorem~6.12]{MasurMinsky00}
\label{Thm:DistEst}
For any surface $S$, there is a constant $\CutOff(S)$ so that for
every $C \geq \CutOff(S)$ there is an $\Err \geq 1$ so that for all
$\mu, \nu \in \calM(S)$
\[
d_{\calM(S)}(\mu, \nu) \,\, =_\Err \,\, \sum [d_X(\mu, \nu)]_C
\]
where the sum ranges over essential subsurfaces $X \subset S$. 
\end{theorem}

\subsection{Marking distance}
Suppose that $\{ \tau_i \}_{i=0}^N \subset \calT(S)$ is a sliding and
splitting sequence.  Let $V_i = V(\tau_i)$ be the set of vertex cycles
of $\tau_i$.  As $i(V_i) \leq k_0$ and $i(V_i, V_{i+1}) \leq \ell_0$
the map $i \mapsto V_i$ gives rise to an edge-path in $\calM(S)$.

Suppose that $[p,q] \subset [0,N]$.  Let $\calS_X(p,q)$ be the set of
indices $r \in [p,q-1]$ so that $\tau_{r+1}|X$ is a splitting of
$\tau_r|X$.  (When $X$ is an annulus $\tau_{r+1}|X$ may also differ
from $\tau_r|X$ by a pair of splits.)  Since slides do not effect
$P(\tau)$ \cite[Proposition 2.2.2]{PennerHarer92} the distance in
$\calM(S)$ between $V_p$ and $V_q$ is at most $|\calS_S(p,q)|$.

\begin{remark}
\label{Rem:Subtrack}
We do not place indices $r$ onto $S_X$ where $\tau_{r+1}|X$ is a
subtrack of $\tau_r|X$; the number of such indices is bounded by a
constant depending only on $X$.
\end{remark}

In the other direction, we must bound the number of splittings between
$\tau_p$ and $\tau_q$ in terms of the marking distance between $V_p$
and $V_q$.  This will be done inductively.  As a notational matter set
$I_S = [0,N]$.  When $X \subset S$ is essential take $V(\tau|X)$ to be
the vertex cycles of the induced track.  Recall that $I_X \subset
I_S$, defined in \refsec{Access}, is the accessible interval for $X
\subset S$.  If $I = [m,n] \subset [0,N]$ is an interval then define
$\calS_X(I) = \calS_X(m,n)$, $d_X(I) = d_X(\tau_m, \tau_n)$, and so
on.

\begin{proposition}
\label{Prop:Induct}
Suppose that $X \subset S$ is an essential subsurface and $J_X \subset
I_X$ is a subinterval.  There is a constant $A = A(X)$, independent of
the sequence $\{ \tau_i \}$, so that $|\calS_X(J_X)| \leq_A
d_{\calM(X)}(J_X).$
\end{proposition}

The rest of this section is devoted to the proof of \refprop{Induct},
from which \refthm{Quasi} follows.

\subsection{Inductive and straight intervals}

We fix two thresholds $\Induct, \Straight$ so that:
\begin{align*}
\max \big\{ 6\NoMotion + 2\SplitSub + 2\Drift(X) + 2, 
        2\Reverse, M_2(X), \CutOff(X) \big\} 
                             &\leq \Induct(X) \\
\max \big\{ \Induct(X) + 2\Reverse, \Branch \SplitSub \big\}
                             &\leq \Straight(X) 
\end{align*}

Here $\NoMotion$ is an upper bound for $d_Y(\alpha, \beta)$ where $Y
\subset S$ is any essential subsurface, $\tau$ is a track, and
$\alpha$ and $\beta$ are wide with respect to $\tau$.  The constant
$\Branch$ is an upper bound for the number of branches in any induced
track.  The constant $\SplitSub$ is an upper bound for the distance
(in any subsurface projection) between the vertices of $\tau$ (or
$\tau|X$) and the vertices of a single splitting or subtrack of
$\tau$.  Also, $M_2(X)$ is the constant provided by Lemma~6.1
of~\cite{MasurMinsky00}.

Recall that the interval $J_X \subset I_X$ is given in
\refprop{Induct}.

\begin{definition}
\label{Def:Induct}
Suppose that $Y \subset X$ is an essential subsurface with $\xi(Y) <
\xi(X)$.  If $d_Y(J_X) \geq \Induct(X)$ then we call $Y$ an {\em
inductive} subsurface of $X$ and take $J_Y = I_Y \cap J_X$ as the
associated {\em inductive subinterval} of $J_X$.  If $d_Y(J_X) <
\Induct(X)$ then we set $J_Y = \emptyset$.
\end{definition}

Suppose $I$ is a subinterval of $J_X$.  Define $\diam_Y(I)$ to be the
diameter, in $\AC(Y)$ of the union $\cup_{i \in I} \pi_Y(\tau_i)$.

\begin{definition}
\label{Def:Straight}
A subinterval $I \subset J_X$ is a {\em straight subinterval} for
$X$ if for all essential subsurfaces $Y \subset X$, with $\xi(Y) <
\xi(X)$, we have $\diam_Y(I) \leq \Straight(X)$.
\end{definition}

\begin{lemma}
\label{Lem:Disjoint}
If $I \subset J_X$ is disjoint from all inductive subintervals of
$J_X$ then $I$ is straight for $X$.
\end{lemma}

\begin{proof}
Fix an essential $Y \subset X$ with $\xi(Y) < \xi(X)$.  It suffices to
show, for every subinterval $J \subset I$, that $d_Y(J) \leq
\Straight(X)$.

If $J \cap I_Y = \emptyset$ then \refthm{FellowTravel} implies $d_Y(J)
\leq \Drift$.
Suppose that $J$ meets $I_Y$; thus $J_Y = \emptyset$ by hypothesis and
so $Y$ is not inductive.  It follows that $d_Y(I) < \Induct(X)$.  By
\reflem{Reverse} we have $d_Y(J) < \Induct(X) + 2\Reverse$.
\end{proof}

\begin{lemma}
\label{Lem:Straight}
There is a constant $A = A(X)$, independent of $\{ \tau_i \}$, so that
if $I \subset J_X$ is straight then $|\calS_X(I)| \leq_A d_X(I)$.
\end{lemma}

\begin{proof}
If $X$ is an annulus then, by \refthm{FellowTravel}, for every $r \in
I$ the core curve $\alpha \subset X$ is carried by and wide in
$\tau_r$.  It follows that the number of switches in $\alpha \subset
\tau_r|X$ is bounded by some constant $K = K(S)$.  Let $q = \max I$
and pick any $\beta \in V(\tau_q|X)$.  As in the proof of
\refthm{LocalUnparam} let $\sigma_r \subset \tau_r|X$ be the minimal
subtrack carrying $\beta$.  Thus $\sigma_r$ has either exactly four
branches and two switches, or is an embedded arc.  It follows that
every $K^2/4$ consecutive splittings in $S_X(I)$ induces at least one
splitting in the sequence of tracks $\{ \sigma_r \}$.  Therefore the
singleton sets $V(\sigma_r)$ form a quasi-geodesic in $\calA(X)$.
Since $V(\sigma_r) \subset V(\tau_r|X)$ the proof is complete when $X$
is an annulus.

We assume for the rest of the proof that $X$ is not an annulus.  The
map $i \mapsto V(\tau_i|X)$, taking tracks to their vertex cycles, is
generally not injective.  (For example, see~\cite[Proposition
2.2.2]{PennerHarer92}.)  However:

\begin{claim}
\label{Clm:NoMotion}
There is a constant $\Inj = \Inj(X)$, independent of $\{ \tau_i \}$,
so that if $V(\tau_r|X) = V(\tau_s|X)$ then $|\calS_X(r,s)| \leq
\Inj$.
\end{claim}

\begin{proof}
Let $\mu = V(\tau_r|X)$.  Our hypothesis on $\tau_s|X$ and induction
proves that $V(\tau_t|X) = \mu$ for all $t \in [r,s]$.  Recurrence and
uniqueness of carrying \cite[Proposition 3.7.3]{Mosher03} implies that
$\tau_{t+1}|X$ is a split or a slide of $\tau_t|X$, and not a
subtrack, for all $t \in [r,s-1]$.

If $t \in [r,s]$ and $b \in \calB(\tau_t|X)$ then define $w_\mu(b) =
\sum_{\alpha \in \mu} w_\alpha(b)$.  Let
\[
M(t) = 
   \big( w_\mu(b) :
         \mbox{$b$ is a large branch of $\tau_t|X$} \big) 
\]
be the sequence of given numbers, arranged in non-increasing order.
Note that if $\tau_{t+1}|X$ is a slide of $\tau_t|X$ then $M(t+1) =
M(t)$.  However, if $t \in \calS_X(r, s)$ then the recurrence of
$\tau_t|X$ implies that $M(t+1) < M(t)$, in lexicographic order.
Since there are only finitely many possibilities for an induced track
$\tau|X$, up to the action of $\MCG(X)$, the claim follows.
\end{proof}




Notice that if $V(\tau_{i+1}|X) \neq V(\tau_i|X)$ then
$V(\tau_{i+1}|X) \neq V(\tau_j|X)$ for $j \leq i$.  This is because
$P(\tau_{k+1}|X) \subset P(\tau_k|X)$ for all $k$.  Using $C = 1 +
\max \{ \CutOff(X), \Straight(X) \}$ as the cut-off in
\refthm{DistEst} gives some constant of quasi-equality, say
$\Err$. Define $\Radius = \Err + 1$.

Suppose that $[p,q] = I$, the straight subinterval of $J_X$ given by
\reflem{Straight}.  We define a function $\rho \from [0,M] \to I$ as
follows: let $\rho(0) = p$ and let $\rho(n+1)$ be the smallest element
in $[\rho(n), q]$ with $d_{\calM(X)}(\tau_{\rho(n)}, \tau_{\rho(n+1)})
= \Radius$.  (If $\rho(n+1)$ is undefined then take $M = n+1$ and
$\rho(M) = q$.)  Let $B(\mu)$ be the ball of radius $\Radius$ about
the marking $\mu \in \calM(X)$.  Define
\[
\Volume = \max \big\{ |B(\mu)| : \mu \in \calM(X) \big\}.
\]
Deduce from \refclm{NoMotion} and the remark immediately following
that, for all $n \in [0, M - 1]$,
\begin{gather*}
|\calS_X(\rho(n), \rho(n+1))| \leq \Inj \Volume. \\
\tag*{\text{Thus}}
|\calS_X(I)| \leq \Inj \Volume \cdot M.
\end{gather*}
So to prove \reflem{Straight} it suffices to bound $M$ from above in
terms of $d_X(I)$.

\begin{claim}
\label{Clm:Jump}
Fix $n \in [0,M-2]$.  Let $\tau, \sigma = \tau_{\rho(n)},
\tau_{\rho(n+1)}$.  Then $d_X(\tau, \sigma) \geq \Reverse + 1$.
\end{claim}

\begin{proof}
We use \refthm{DistEst}.  Note that $d_{\calM(X)}(\tau, \sigma) =
\Radius$.  Since $\Radius$ is greater than the additive error there is
at least one non-vanishing term in the sum $\sum_{Y \subset X}
[d_Y(\tau, \sigma)]_C$.  

However, since $[\rho(n), \rho(n+1)] \subset [p,q]$ and $[p,q] = I$ is
straight we have $d_Y(\tau, \sigma) \leq \Straight(X)$ for all $Y
\subset X$ with $\xi(Y) < \xi(X)$.  Thus $d_X(\tau, \sigma)$ is the
only term of the sum greater than the cut-off $C$.  Since $C >
\Straight(X) \geq \Reverse$, we have $d_X(\tau, \sigma) \geq \Reverse
+ 1$ and the claim is proved.
\end{proof}

Thus we have
\begin{align*}
d_X(I)    &\geq -(M - 1) \cdot \Reverse + 
             \sum_{n = 0}^{M - 1} d_X(\tau_{\rho(n)}, \tau_{\rho(n+1)}) \\
          &\geq M - 1 + d_X(\tau_{\rho(M-1)}, \tau_{\rho(M)}) \\
          &\geq M - 1 
\end{align*}
where the first and second lines follow from \reflem{Reverse} and
\refclm{Jump} respectively.  This completes the proof of
\reflem{Straight}.
\end{proof}

\begin{lemma}
\label{Lem:StraightInInduct}
There is a constant $A = A(X)$ with the following property: Suppose
that $J_Y \subset J_X$ is an inductive interval.  Suppose that $I
\subset J_Y$ is a straight subinterval for $X$.  Then $|\calS_X(I)|
\leq A$.
\end{lemma}

\begin{proof}
Let $[p,q] = I$.  Applying \refthm{FellowTravel}, as $p \in J_Y
\subset I_Y$, the multicurve $\bdy Y$ is wide with respect to
$\tau_p$.  It follows that $\bdy Y$ is also wide with respect to
$\tau_p^X$.  Note that the curves of $V(\tau_p|X)$ are also wide with
respect to $\tau_p^X$.  Repeating this discussion for $q$, and then
applying \reflem{GeneralBound} and the triangle inequality gives a
uniform bound for $d_X(\tau_p, \tau_q)$.  The lemma now follows from
\reflem{Straight}.
\end{proof}

\subsection{Long and short intervals}

\begin{definition}
\label{Def:Short}
A straight subinterval $I$ for $X$ is {\em short} if $d_X(I) \leq
4\Reverse$.  Otherwise $I$ is {\em long}.
\end{definition}

By \reflem{Straight}, if $I$ is a short straight interval then
$|\calS_X(I)|$ is uniformly bounded by a constant depending only on
$X$.

Let $\Ind$ be the set of inductive subsurfaces $Y \subset X$.  Define
$\Ind' = \Ind \cup \{ X \}$.  Note that every maximal subinterval of
$J_X \setminus \cup_{Y \in \Ind} J_Y$ is straight, by
\reflem{Disjoint}.  We partition these maximal subintervals into the
sets $\Long$ and $\Short$ as the given interval is long or short
respectively.

\begin{lemma}
\label{Lem:LongBound}
There is a constant $A = A(X)$, independent of $\{ \tau_i \}$, so that
\[
\sum_{I \in \Long} |\calS_X(I)| \leq_A d_X(J_X).
\]
\end{lemma}

\begin{proof}
From \reflem{Straight} we deduce
\[
\sum_{I \in \Long} |\calS_X(I)| 
   \leq_A |\Long| + \sum_{I \in \Long} d_X(I) 
\]
where the first term on the right arises from addition of additive
errors.  
By the definition of a long straight interval and from
\reflem{Reverse} deduce
\[
4\Reverse|\Long| \leq \sum_{I \in \Long} d_X(I) 
     \leq d_X(J_X) + 2\Reverse|\Long|.
\]
Thus $2\Reverse|\Long| \leq d_X(J_X)$.  These inequalities combine to
prove the lemma, for a somewhat larger value of $A = A(X)$.
\end{proof}

\begin{lemma}
\label{Lem:ShortBound}
There is a constant $A = A(X)$, independent of $\{ \tau_i \}$, so that
\[
\sum_{I \in \Short} |\calS_X(I)| \leq_A |\Ind'|.
\]
\end{lemma}

\begin{proof}
By \reflem{Straight} the number of splittings in any short straight
interval is a priori bounded (depending only on $X$).  Since $|\Short|
\leq |\Ind'|$ the lemma follows. 
\end{proof}

\begin{lemma}
\label{Lem:Nested}
If $Z \in \Ind$ then 
\[
\card \big\{ Y \in \Ind \st Z \subset Y,\,\xi(Z) < \xi(Y) \big\} 
    \leq 2(\xi(X) - \xi(Z) - 1).
\]
\end{lemma}

\noindent
This follows from and is strictly weaker than Theorem~4.7 and
Lemma~6.1 of~\cite{MasurMinsky00}.  We give a proof, using our
structure theorem, to extract the necessary lower bound for
$\Induct(X)$.

\begin{proof}[Proof of \reflem{Nested}]
Suppose that $U \in \Ind$ contains $Z$.  Suppose $J_Z = [p,q]$ and
$J_X = [m,n]$.  Thus $\bdy Z$ is wide with respect to $\tau_p$.  So
$d_U(\tau_p, \bdy Z) \leq \NoMotion$, by the definition of
$\NoMotion$, and the same holds at the index $q$.  Thus $d_U(\tau_p,
\tau_q) \leq 2\NoMotion$.  The subsurface $U$ {\em precedes} or {\em
succeeds} $Z$ if $d_U(\tau_m, \tau_p)$ or $d_U(\tau_q, \tau_n)$,
respectively, is greater than or equal to $2\NoMotion + \SplitSub +
\Drift(X) + 1$.  Note that $U$ must precede or succeed $Z$ (or both)
as otherwise $d_U(\tau_m, \tau_n) < 6\NoMotion + 2\SplitSub +
2\Drift(X) + 2 \leq \Induct(X)$, a contradiction.

It now suffices to consider subsurfaces $U$ and $V$ that both succeed
and both contain $Z$.  If $\max J_U \leq \max J_V$ then $U \subset V$.
For, if not, $\bdy V$ cuts $U$ while missing $Z$.  Since $\bdy V$ is
wide at the index $r = \max J_V$ we deduce that
\begin{align*}
d_U(\tau_q, \tau_n) 
  \leq {} & d_U(\tau_q, \bdy Z) + d_U(\bdy Z, \bdy V) 
            + d_U(\bdy V, \tau_r) + \\
          & + d_U(\tau_r, \tau_{r+1}) + d_U(\tau_{r+1}, \tau_n) \\
  \leq {} & 2\NoMotion + \SplitSub + \Drift(X) + 1
\end{align*}
and this is a contradiction.  Thus the surfaces in $\Ind$ that
strictly contain $Z$, and succeed $Z$, are nested.  
\end{proof}

\begin{definition}
\label{Def:Assign}
{\em Assign} an index $r \in \calS_X(J_X)$ to a subsurface $Y \subset
X$ if $Y \in \Ind'$, $r \in J_Y$, $\tau_{r+1}|Y$ is a splitting of
$\tau_r|Y$ and there is no subsurface $Z \subset Y, \xi(Z) < \xi(Y)$
with those three properties.  
\end{definition}

\begin{lemma}
\label{Lem:InductiveBound}
There is a constant $A = A(X)$, independent of $\{ \tau_i \}$, so that
the number of splittings contained in inductive intervals is
quasi-bounded by $|\Ind| + \sum_{Y \in \Ind} d_Y(J_X)$. 
\end{lemma}


\begin{proof}
Fix $Y \in \Ind$.  Consider an index $r \in J_Y$ that is assigned to
$X$.  Let $I \subset J_Y$ be the maximal interval containing $r$ so
that all indices in $\calS_X(I)$ are assigned to $X$.  We now show
that $I$ is straight.  Let $Z$ be any essential subsurface of $X$ with
$\xi(Z) < \xi(X)$ and let $[r,s] = J \subset I$ be any subinterval.
If $J \cap I_Z = \emptyset$ then \refthm{FellowTravel} implies that
$d_Z(J) \leq \Drift$.  If $J$ meets $I_Z$ then, as no splittings of
$J$ are assigned to $Z$ we deduce that $\tau_s|Z$ is obtained from
$\tau_r|Z$ by sliding and taking subtracks only.  Thus $d_Z(J) \leq
\Branch \SplitSub \leq \Straight(X)$, as desired.

By \reflem{StraightInInduct} we find that $|\calS_X(I)|$ is
bounded.  It follows that the number of splittings in the inductive
intervals is quasi-bounded by $\sum_{Y \in \Ind} |\calS_Y(J_Y)|$.

By induction, \refprop{Induct} gives 
\[
|\calS_Y(J_Y)| \leq_A d_{\calM(Y)}(J_Y).  
\]
Taking a cutoff of $C = 1 + \max \{ \CutOff(Y), \Induct(X) + 2\Reverse
\}$ and applying the distance estimate \refthm{DistEst} we have a
quasi-inequality
\[
d_{\calM(Y)}(J_Y) \, \leq_\Err \sum_{Z \subset Y} [d_Z(J_Y)]_C.
\]
Since $d_Z(J_Y) \leq d_Z(J_X) + 2\Reverse$ for all $Z \subset Y$, it
follows that non-zero terms in the sum only arise for subsurfaces in
$\Ind'(Y) = \{ Z \in \Ind' \st Z \subset Y \}$.  Since $2\Reverse \leq
\Induct(X)$
we have $[d_Z(J_Y)]_C \leq 2\cdot d_Z(J_X)$.  Making $A = A(X)$ larger
if necessary we have
\begin{align*}
|\calS_Y(J_Y)| 
  &\leq_A \sum_{Z \in \Ind'(Y)} d_Z(J_X). \\
\tag*{\text{Thus}}
\sum_{Y \in \Ind} |\calS_Y(J_Y)| 
  &\leq_A |\Ind| + \sum_{Y \in \Ind} 
                       \sum_{Z \in \Ind'(Y)} d_Z(J_X) \\
  &\leq_{A} |\Ind| + \sum_{Y \in \Ind} d_Y(J_X) 
\end{align*}
where the final quasi-inequality follows from \reflem{Nested}, taking
$A$ larger as necessary.  Note that the term $|\Ind|$ on the middle
line arises by adding additive errors.  This proves
\reflem{InductiveBound}.
\end{proof}

Since every index in $\calS_X(J_X)$ is either in a long or short
straight interval or in an inductive interval, from Lemmas
\ref{Lem:LongBound}, \ref{Lem:ShortBound}, and
\ref{Lem:InductiveBound} and increasing $A$ slightly, we have:
\[
|\calS_X(J_X)| \leq_A  d_X(J_X) + |\Ind'| + 
                \sum_{Y \in \Ind} d_Y(J_X). 
\]
Note that $|\Ind'| \leq_A d_{\calM(X)}(J_X)$; this follows
from the hierarchy machine (in particular Lemma~6.2 and Theorem~6.10
of~\cite{MasurMinsky00}) and because $\Induct(X) \geq M_2(X)$,
the constant of Lemma~6.1 in \cite{MasurMinsky00}.  Finally,
\[
\sum_{Y \in \Ind'} d_Y(J_X) \leq_A d_{\calM(X)}(J_X)
\]
follows from the distance estimate (\refthm{DistEst}) and because
$\Induct(X) \geq \CutOff(X)$.
This completes the proof of \refprop{Induct} and thus the proof of
\refthm{Quasi}. \qed

\bibliographystyle{plain}
\bibliography{bibfile}
\end{document}